\documentclass[12pt,reqno]{amsart}
\usepackage{amsfonts}
\usepackage{a4wide}

\numberwithin{equation}{section}

\newtheorem{theorem}{Theorem}[section]
\newtheorem{proposition}[theorem]{Proposition}
\newtheorem{corollary}[theorem]{Corollary}
\newtheorem{lemma}[theorem]{Lemma}

\theoremstyle{definition}

\newtheorem{definition}[theorem]{Definition}

\theoremstyle{remark}
\newtheorem{remark}[theorem]{Remark}

\begin{document}

\title[ Nonscattering solutions
 ]
{
   Nonscattering  solutions to the
    $L^{2}$ -supercritical
    NLS Equations
 }

 \author{Qing Guo
}

\address{Academy of Mathematics and Systems Science, Chinese Academy of Sciences, Beijing 100190, P.R. China}

\email{guoqing@amss.ac.cn}





\begin{abstract}
We investigate the nonlinear Schr\"{o}dinger equation
 $iu_{t}+\Delta u+|u|^{p-1}u=0$ with $1+\frac{4}{N}<p<1+\frac{4}{N-2}$ (when $N=1, 2$, $1+\frac{4}{N}<p<\infty$ )
 in energy space $H^1$
and study the divergent property of infinite-variance and nonradial
  solutions.
If $M(u)^{\frac{1-s_{c}}{s_{c}}}E(u)<M(Q)^{\frac{1-s_{c}}{s_{c}}}E(Q)$ and
 $\|u_{0}\|_{2}^{\frac{1-s_{c}}{s_{c}}}\|\nabla u_{0}\|_{2}>\|Q\|_{2}^{\frac{1-s_{c}}{s_{c}}}\|\nabla Q\|_{2},$
then either  $u(t)$~blows up in finite forward time,  or $u(t)$ exists globally for positive time and there exists a
time sequence $t_{n}\rightarrow+\infty$ such that  $\|\nabla u(t_{n})\|_{2}\rightarrow+\infty.$ Here $Q$ is the ground state
solution of $-Q+\Delta Q+|Q|^{p-1}Q=0.$
A similar result holds for negative time. This extend the result of
the 3D cubic Schr\"{o}dinger equation in \cite{holmer10}  to the general mass-supercritical and energy-subcritical case .

\end{abstract}

\maketitle
MSC: 35Q55, 35A15,
35B30.\\

Keywords: Nonlinear Schr\"{o}dinger equation;  Blow-up solution; Infinite variance;
Mass-supercritical; Energy-subcritical

\section{Introduction}

We consider the following Cauchy problem of a nonlinear Schr\"{o}dinger equation 

\begin{equation}\label{1.1}
\left\{ \begin{aligned}
         \ iu_{t}+\Delta u+|u|^{p-1}u&=0,\ \ \ (x,t)\in R^{N}\times R, \\
                  \ u(x,0)&=u_{0}(x)\in H^{1}(R^{N}).
                          \end{aligned}\right.
                          \end{equation}

It is well known from Ginibre and Velo ~\cite{JG1}~that, equation~\eqref{1.1}~is locally well-posed in~$ H^{1}.$~
That is for ~$u_{0}\in H^{1},$~there exist~$0<T\leq\infty$~and a unique solution~$u(t)\in C([0,T);H^{1})$~to  ~\eqref{1.1}.~
When ~$T=\infty,$~  we say that the solution is positively  global; while on the other hand,  we  have
~$\lim_{t\uparrow T}\|\nabla u(t)\|_{2}\rightarrow\infty$~  and call that this solution blows up in finite positive time.
Solutions of
~\eqref{1.1}~admits the following conservation laws in energy space  ~$H^{1}:$~
\begin{align*}
L^{2}-norm:\ \ \ \ M(u)(t)&\equiv \int|u(x,t)|^{2}dx=M(u_{0});\\
Energy:\ \ \ \ E(u)(t)&\equiv \frac{1}{2}\int|\nabla u(x,t)|^{2}dx-\frac{1}{p+1}\int|u(x,t)|^{p+1}dx=E(u_{0});\\
Momentum:\ \ \ \ P(u)(t)&\equiv Im\int\overline{u}(x,t)\nabla u(x,t)dx=P(u_{0}).
\end{align*}

Note that equation~\eqref{1.1}~is invariant under the scaling~$u(x,t)\rightarrow\lambda^\frac{2}{p-1}u(\lambda x,\lambda^2t)$~
which also leaves the homogeneous Sobolev norm~
$\dot{H}^{s_c}$ ~invariant with~ $s_c=\frac{N}{2}-\frac{2}{p-1}.$~
It is classical from the conservation of the energy and the~$L^{2}$~norm that for~ $s_c<0$,~the equation is subcritical
and all ~$ H^{1}$~solutions are global and ~$ H^{1}$~bounded. The smallest power for which blow up may occur is~$p=1+\frac{4}{N}$~
which is referred to as the~$L^{2}$~critical case  corresponding to~ $s_c=0$~\cite{RTG}~\cite{SK2}. The case~ $0<s_c<1$~is called the
~$L^{2}$~supercritical and ~$H^{1}$~subcritical or the Mass-supercritical and Energy-subcritical case. In fact, we are concerning in this paper with the case~ $0<s_c<1.$~

For the 3D cubic nonlinear Schr\"{o}dinger equation
with~ $s_c=\frac{1}{2}$~and~$p=3,$~there have been several results on  either scattering  or blow-up solutions.
In Holmer and Roudenko
~\cite{radial},~the authors proved that if ~$u_{0}\in H^{1}$~is radial, $M(u)E(u)<M(Q)E(Q)$
~and
~$\|\nabla u_{0}\|_{2}\|u_{0}\|_{2}<\|\nabla Q\|_{2}\|Q\|_{2},$~
then the solution ~$u(t)$~is globally well-posed and scattering; They further showed that
if~$M(u)E(u)<M(Q)E(Q)$
~and
~$\|\nabla u_{0}\|_{2}\|u_{0}\|_{2}>\|\nabla Q\|_{2}\|Q\|_{2},$~then the solution blows up in finite time,
provided that either the initial data has finite variance or is radial. The radial case is an extension of a result of
Ogawa and Tsutsumi~\cite{OT} ~who proved the case~$E(u)<0.$~Then in~\cite{nonradial},~also for the 3D cubic nonlinear Schr\"{o}dinger equation,
  the authors extended the
 scattering results on  radial~$ H^{1}$~solutions to the nonradial case . The technique employed is parallel to that
 employed by Kenig-Merle~\cite{km}~in their study of the energy-critical NLS. For  ~ $0<s_c<1$~, ~\cite{yuan}~
 have extended the scattering results to the general~$L^{2}$~supercritical and ~$H^{1}$~subcritical case.

Then in Holmer and Roudenko
~\cite{holmer10}, the authors further studied the blow-up theory for the 3D cubic nonlinear Schr\"{o}dinger equation , which
 dropped  the additional hypothesis of finite variance and radiality. More precisely, they proved that
 if~$M(u)E(u)<M(Q)E(Q)$
~and
~$\|\nabla u_{0}\|_{2}\|u_{0}\|_{2}>\|\nabla Q\|_{2}\|Q\|_{2},$~then either ~$u(t)$~blows up in finite positive time,
or~$u(t)$~exists globally for all positive time and there exists a time sequence~$t_{n}\rightarrow+\infty$~such that ~$\|\nabla u(t_{n})\|_{2}\rightarrow\infty,$~with  similar results holding  for negative time.

In this paper, we extend the above results to the general~$L^{2}$~supercritical and ~$H^{1}$~subcritical case,
and obtain the following conclusion:

\begin{theorem}\label{th1}

Suppose ~$u_{0}\in H^{1},$~~$M(u)^{\frac{1-s_{c}}{s_{c}}}E(u)<M(Q)^{\frac{1-s_{c}}{s_{c}}}E(Q)$~and\\
~$\|\nabla u_{0}\|_{2}\|u_{0}\|_{2}^{\frac{1-s_{c}}{s_{c}}}>\|\nabla Q\|_{2}\|Q\|_{2}^{\frac{1-s_{c}}{s_{c}}}.$~
Then either ~$u(t)$~blows up in finite forward time,  or~$u(t)$~is forward global and there exists a
time sequence~$t_{n}\rightarrow\infty$~such that ~$\|\nabla u(t_{n})\|_{2}\rightarrow\infty.$~
A similar statement holds for negative time.
\end{theorem}

Different from a similar result obtained by Glangeta and Merle~\cite{gm}~for the case~$E(u)<0,$~
our proof is by means of the profile decomposition introduced by Keraani~\cite{keraani},~nonlinear pertubation
theory based on the Strichartz estimate~\cite{TC}~~\cite{KT},~and the rigidity theorems based upon the localized virial identity.
Though with the same idea as in \cite{holmer10},  we still have to
reestablish  the tools mentioned  above, such as the profile decomposition,  to conquer the difficulties  our general case should bring .

\begin{remark}\label{p}
Via the Galilean transform and momentum conservation, in this paper, we will always assume
that~$P(u)=0,$~and  put further standard details in the Appendix. That is to say we need only show Theorem~\ref{th1}~under
the condition ~$P(u)=0.$~
\end{remark}
In this paper,
we denote the
Sobolev space $ H^{1}(\mathbb{R}^{N})$ as $H^{1}$ for short, and the ~$L^{p}$~norm~$\|\cdot\|_{p}.$~
Also for convenience, we will use the notation $C,$ except for some specifications,  standing for the variant absolute constants.

After this paper was finished, we learnt that the authors in
\cite{version} has obtained the same result presented in this paper.
However, the proof here is different from that in \cite{version}. We
hope our method can be helpful for other related problem.

\section{Preliminaries }

In this section,  we will review some basic facts about the ground state and give a dichotomy result.

Weinstein in \cite{W1} proved that the sharp constant ~$C_{GN}$~ of Gagliardo-Nirenberg inequality for $0<s_{c}<1$

\begin{equation}\label{2.1}
\|u\|^{p+1}_{L^{p+1}(R^{N})}\leq C_{GN}\|\nabla
u\|_{L^{2}(R^{N})}^{\frac{N(p-1)}{2}}\|u\|_{L^{2}(R^{N})}^{2-\frac{(N-2)(p-1)}{2}}
\end{equation}
is achieved by ~$u=Q,$~ where ~$Q$~ is the ground state of
$$-(1-s_{c})Q+ \Delta Q+|Q|^{p-1}Q=0.$$
Using Pohozhaev identities we can get the following identities without difficulty:
$$\|Q\|_{2}^{2}=\frac{2}{N}\|\nabla Q\|_{2}^{2},$$
$$\|Q\|^{p+1}_{p+1}=\frac{2(p+1)}{N(p-1)}\|\nabla Q\|_{2}^{2}=\frac{(p+1)}{(p-1)}\|Q\|_{2}^{2},$$
\begin{equation}\label{2.2}
E(Q)=\frac{N(p-1)-4}{2N(p-1)}\|\nabla Q\|_{2}^{2}=\frac{N(p-1)-4}{4(p-1)}\|Q\|_{2}^{2}=\frac{N(p-1)-4}{4(p+1)}\|Q\|^{p+1}_{p+1},
\end{equation}
and $C_{GN}$  can be  expressed by
\begin{equation}\label{2.3}
 C_{GN}=\frac{\|Q\|^{p+1}_{p+1}}{\|\nabla Q\|_{2}^{\frac{N(p-1)}{2}}\|Q\|_{2}^{2-\frac{(N-2)(p-1)}{2}}}.
 \end{equation}

Note that the Sobolev ~$\dot{H}^{s_{c}}$~norm and the equation~\eqref{1.1} are invariant under the scaling
~$u(x,t)\mapsto u_{\lambda}(x,t)=\lambda^{\frac{2}{p-1}}u(\lambda x,\lambda^{2}t).$
~Other scaling invariant quantities are ~$\|\nabla u\|_{2}\|u\|_{2}^{\frac{1-s_{c}}{s_{c}}}$~\\
and ~$E(u)M(u)^{\frac{1-s_{c}}{s_{c}}}.$~

Let
\begin{equation}\label{eta}
\eta(t)=\frac{\|\nabla u\|_{2}\|u\|_{2}^{\frac{1-s_{c}}{s_{c}}}}{\|\nabla Q\|_{2}\|Q\|_{2}^{\frac{1-s_{c}}{s_{c}}}}.
\end{equation}
In order to study the relationship between~$\eta(t)$~and~$\frac{E(u)M(u)^{\frac{1-s_{c}}{s_{c}}}}{E(Q)M(Q)^{\frac{1-s_{c}}{s_{c}}}}
,$~
we might as well assume ~$\|u\|_{2}=\|Q\|_{2}$~by scaling. Denote~$\omega_{1}=\frac{N(p-1)}{N(p-1)-4}$~and
~$\omega_{2}=\frac{4}{N(p-1)-4}.$~
Then by \eqref{2.1}-\eqref{2.3} we have

\begin{align*}
&2\omega_{1}\frac{\|\nabla u\|_{2}^{2}\|u\|_{2}^{\frac{2-2s_{c}}{s_{c}}}}{\|\nabla Q\|_{2}^{2}\|Q\|_{2}^{\frac{2-2s_{c}}{s_{c}}}}
\geq\frac{E(u)M(u)^{\frac{1-s_{c}}{s_{c}}}}{E(Q)M(Q)^{\frac{1-s_{c}}{s_{c}}}}
=\frac{E(u)}{E(Q)}\\
&=\omega_{1}\frac{\|\nabla u\|_{2}^{2}}{\|\nabla Q\|_{2}^{2}}-\frac{2\omega_{1}}{p+1}\frac{\|u\|^{p+1}_{p+1}}{\|
 Q\|^{p+1}_{p+1}}\\
&\geq \omega_{1}\frac{\|\nabla u\|_{2}^{2}}{\|\nabla Q\|_{2}^{2}}-\frac{2\omega_{1}}{p+1}\frac{C_{GN}\|\nabla u\|_{2}^{\frac{N(p-1)}{2}}\|u\|_{2}^{2-\frac{(N-2)(p-1)}{2}}}{\|Q\|^{2}_{2}}\\
&=\omega_{1}\eta(t)^{2}-\frac{2\omega_{1}}{p+1}\frac{C_{GN}\|Q\|_{2}^{2-\frac{(N-2)(p-1)}{2}}}{\|\nabla Q\|_{2}^{2-\frac{(N-2)(p-1)}{2}}}\frac{\|\nabla u\|_{2}^{\frac{N(p-1)}{2}}}{\|\nabla Q\|_{2}^{\frac{N(p-1)}{2}}}\\
&=\omega_{1}\eta(t)^{2}-\frac{4\omega_{1}}{N(p-1)}\eta(t)^{\frac{N(p-1)}{2}}
=\omega_{1}\eta(t)^{2}-\omega_{2}\eta(t)^{\frac{N(p-1)}{2}}.
\end{align*}
That is
\begin{equation}\label{2.5}
2\omega_{1}\eta(t)^{2}
\geq \frac{E(u)M(u)^{\frac{1-s_{c}}{s_{c}}}}{E(Q)M(Q)^{\frac{1-s_{c}}{s_{c}}}}
\geq \omega_{1}\eta(t)^{2}-\omega_{2}\eta(t)^{\frac{N(p-1)}{2}}.
\end{equation}

Note that ~$\frac{\omega_{1}}{\omega_{2}}>1$~as ~$\frac{4}{N}<p-1<\frac{4}{N-2}.$~
Thus it is not difficult to observe that if ~$0\leq M(u)^{\frac{1-s_{c}}{s_{c}}}E(u)/ M(Q)^{\frac{1-s_{c}}{s_{c}}}E(Q)<1,$~
then there exist two solutions~$0\leq \lambda_{-}<1<\lambda $~to the following equation of ~$\lambda $~
\begin{equation}\label{2.6}
 \frac{E(u)M(u)^{\frac{1-s_{c}}{s_{c}}}}{E(Q)M(Q)^{\frac{1-s_{c}}{s_{c}}}}
= \omega_{1}\lambda^{2}-\omega_{2}\lambda^{\frac{N(p-1)}{2}}.
\end{equation}

By the ~$H^{1}$~local theory \cite{TC}~, there exist~$-\infty\leq T_{-}<0<T_{+}\leq+\infty$~such that~$(T_{-},T_{+})$~is the maximal
time interval of existence for ~$u(t)$~solving~\eqref{1.1}~, and if~$T_{+}<+\infty$~then
$$\|\nabla u(t)\|_{2}\geq\frac{C}{(T_{+}-t)^{\frac{1}{p-1}-\frac{N-2}{4}}}\ \ \ as~t\uparrow T_{+},$$
and a similar argument holds if ~$-\infty< T_{-}.$~ Moreover,
as a consequence of the continuity of the flow~$u(t),$~we have the following dichotomy proposition :
\begin{proposition}\label{p21'}
(Global versus blow-up dichotomy)~Let~ $u_{0}\in H^{1}(R^{N})$,~ and let ~$I=(T_{-},T_{+})$ ~be the
maximal time interval of existence of ~$u(t)$~ solving ~\eqref{1.1}.~
Suppose that
\begin{equation}\label{2.1'}
M(u)^{\frac{1-s_{c}}{s_{c}}}E(u)<M(Q)^{\frac{1-s_{c}}{s_{c}}}E(Q).
\end{equation}
If ~\eqref{2.1'}~holds and
\begin{equation}\label{2.2'}
\|u_{0}\|_{2}^{\frac{1-s_{c}}{s_{c}}}\|\nabla u_{0}\|_{2}<\|Q\|_{2}^{\frac{1-s_{c}}{s_{c}}}\|\nabla Q\|_{2},
\end{equation}
then ~$I=(-\infty,+\infty)$,~ i.e., the solution exists globally in
time, and for all time ~$t\in \mathbb{R},$~
\begin{equation}\label{2.3'}
\|u(t)\|_{2}^{\frac{1-s_{c}}{s_{c}}}\|\nabla u(t)\|_{2}<\|Q\|_{2}^{\frac{1-s_{c}}{s_{c}}}\|\nabla Q\|_{2}.
\end{equation}
If ~\eqref{2.1'}~holds and
\begin{equation}\label{2.4'}
\|u_{0}\|_{2}^{\frac{1-s_{c}}{s_{c}}}\|\nabla u_{0}\|_{2}>\|Q\|_{2}^{\frac{1-s_{c}}{s_{c}}}\|\nabla Q\|_{2},
\end{equation}
then for~$t\in I,$~
\begin{equation}\label{2.5'}
\|u(t)\|_{2}^{\frac{1-s_{c}}{s_{c}}}\|\nabla u(t)\|_{2}>\|Q\|_{2}^{\frac{1-s_{c}}{s_{c}}}\|\nabla Q\|_{2}.
\end{equation}

\end{proposition}
\begin{proof} Multiplying the formula of energy by
$M(u)^{\frac{1}{s_{c}}-1}$ and using the Gagliardo-Nirenberg
inequality we have
\begin{eqnarray}
E(u)M(u)^{\frac{1}{s_{c}}-1}&=&\frac{1}{2}\|\nabla
u\|_{L^2}^2\|u\|_{L^2}^{\frac{2}{s_{c}}-2}-\frac{1}{p+1}\|u\|_{L^{p+1}}^{p+1}\|u\|_{L^2}^{\frac{2}{s_{c}}-2}\nonumber\\
&\geq&\frac{1}{2}(\|\nabla
u\|_{2}\|u\|_{2}^{\frac{1-s_{c}}{s_{c}}})^{2}-\frac{1}{p+1}C_{GN}(\|\nabla
u\|_{2}\|u\|_{2}^{\frac{1-s_{c}}{s_{c}}})^{\frac{N(p-1)}{2}}.\nonumber
\end{eqnarray}
Define ~$f(x)=\frac{1}{2}x^{2}-\frac{1}{p+1}
C_{GN}x^{\frac{N(p-1)}{2}}.$~ Since ~$N(p-1)\geq4$,~ then~
$f'(x)=x(1-C_{GN}\frac{N(p-1)}{2(p+1)}x^{\frac{N(p-1)-4}{2}}),$~\\
and~ $f'(x)=0$~ when~ $x_{0}=0$~ and ~$x_{1}=\|\nabla
Q\|_{2}\|Q\|_{2}^{\frac{1-s_{c}}{s_{c}}}$.~ Note that~ $f(0)=0$ ~and
~$f(x_{1})=E(u)M(u)^{\frac{1}{s_{c}}-1},$~ thus the graph of~ $f$ ~has
two extrema: a local minimum at~ $x_{0}$ ~and a local maximum at
~$x_{1}$. ~The condition ~\eqref{2.1'}~implies that~
$E(u_{0})M(u_{0})^{\frac{1}{s_{c}}-1}<f(x_{1})$.~ Combining with
energy conservation, we have
\begin{equation}\label{2.6'}
f(\|\nabla
u\|_{2}\|u\|_{2}^{\frac{1-s_{c}}{s_{c}}})\leq E(u)M(u_{0})^{\frac{1}{s_{c}}-1}
=E(u)M(u)^{\frac{1}{s_{c}}-1}<f(x_{1}).
\end{equation}

 If initially~
$\|\nabla u_{0}\|_{2}\|u_{0}\|_{2}^{\frac{1-s_{c}}{s_{c}}}<x_{1}$,~i.e., the condition~\eqref{2.2'}~holds, then by~\eqref{2.6'}~
and the continuity of ~$\|\nabla u(t)\|_{2}$~in~$t,$~
we have~ $\|\nabla
u(t)\|_{2}\|u(t)\|_{2}^{\frac{1-s_{c}}{s_{c}}}<x_{1}$~for all time~$t\in I.$~In particular, the~$H^{1}$~norm of the solution is bounded,
which implies the global existence and~\eqref{2.3'}~ in this case.

 If initially~
$\|\nabla u_{0}\|_{2}\|u_{0}\|_{2}^{\frac{1-s_{c}}{s_{c}}}>x_{1}$,~i.e., the condition~\eqref{2.4'}~holds, then by~\eqref{2.6'}~
and the continuity of ~$\|\nabla u(t)\|_{2}$~in~$t,$~
we have~ $\|\nabla
u(t)\|_{2}\|u(t)\|_{2}^{\frac{1-s_{c}}{s_{c}}}>x_{1}$~for all time~$t\in I,$~which proves~\eqref{2.5'}.

\end{proof}

The following is another statement of the dichotomy proposition in terms of
 $\lambda$ and $\eta(t)$ defined by \eqref{2.6} and \eqref{eta} respectively, which will be useful in the sequel.

\begin{proposition}\label{p21}
Let~$M(u)^{\frac{1-s_{c}}{s_{c}}}E(u)<M(Q)^{\frac{1-s_{c}}{s_{c}}}E(Q)$~and ~$0\leq \lambda_{-}<1<\lambda $~be defined by \eqref{2.6}. Then exactly one of the following holds:\\
(1)\ \ The solution~$u(t)$~to\eqref{1.1} is global and
$$\frac{1}{2\omega_{1}} \frac{E(u)M(u)^{\frac{1-s_{c}}{s_{c}}}}{E(Q)M(Q)^{\frac{1-s_{c}}{s_{c}}}}
\leq \eta(t)^{2}\leq\lambda_{-}^{2},\ \ \ \forall~ t \in (-\infty,+\infty)$$
(2)\ \ ~$1<\lambda^{2}\leq \eta(t)^{2},$~~$\forall ~t \in (T_{-},T_{+}).$~

\end{proposition}

Naturally, whether the solution is of the first or second type in Proposition~\ref{p21}~is determined by checking the initial data.
Note that the second case does not assert finite-time blow-up. In the first case, we have further results as follows , the proof of which
is almost the same as~\cite{yuan}.

\begin{lemma}\label{sd}
(Small initial data). Let $\|u_{0}\|_{\dot{H}^{s_{c}}}\leq A$, then
there exists  $\delta_{sd}=\delta_{sd}(A)>0$ such that if
~$\|e^{it\Delta}u_{0}\|_{S(\dot{H}^{s_{c}})}\leq \delta_{sd}, $~then~$ u $~solving \eqref{1.1} is global and
\begin{eqnarray}
&&\|u\|_{S(\dot{H}^{s_{c}})}\leq
2\|e^{it\Delta}u_{0}\|_{S(\dot{H}^{s_{c}})},\\
&&\|D^{s_{c}}u\|_{S(L^{2})}\leq
2c\|u_{0}\|_{\dot{H}^{s_{c}}}.
\end{eqnarray}
(one will find~$\|\cdot\|_{S(\dot{H}^{s_{c}})}$~in Section 6, and note that by
 Strichartz estimates, the hypotheses are satisfied if
~$\|u_{0}\|_{\dot{H}^{s_{c}}}\leq C\delta_{sd}. $)
\end{lemma}

\begin{lemma}\label{wave}
(Existence of wave operators). Suppose that $\psi^{+}\in H^1$ and
\begin{equation}\label{3.6}
\frac{1}{2^{}}||\nabla\psi^{+}||_{2}^{2}M(\psi^{+})^{\frac{1-s_{c}}{s_{c}}}<
E(Q)^{}M(Q)^{\frac{1-s_{c}}{s_{c}}}.
\end{equation}
Then there exists $v_{0}\in H^1$ such that v solves \eqref{1.1} with
initial data $v_0$ globally in $H^{1}$ with
$$\|\nabla
v(t)\|_{2}^{}\|v_{0}\|_{2}^{\frac{1-s_{c}}{s_{c}}}<\|\nabla
Q\|_{2}^{}\|Q\|_{2}^{\frac{1-s_{c}}{s_{c}}},M(v)=\|\psi^{+}\|_{2}^{2},E[v]=\frac{1}{2}\|\nabla\psi^{+}\|_{2}^{2},$$
and
$$\lim_{t\rightarrow+\infty}\|v(t)-e^{it\Delta}\psi^{+}\|_{H^1}=0.$$
Moreover, if
$\|e^{it\Delta}\psi^{+}\|_{S(\dot{H}^{s_{c}})}\leq \delta_{sd}$, then
$$\|v_{0}\|_{\dot{H}^{s_{c}}}\leq2\|\psi^{+}\|_{\dot{H}^{s_{c}}}\
\mathrm{and} \
 \|v\|_{S(\dot{H}^{s_{c}})}\leq2\|e^{it\Delta}\psi^{+}\|_{S(\dot{H}^{s_{c}})}.$$
$$\|D^{s}v\|_{S(L^{2})}\leq c\|\psi^{+}\|_{\dot{H}^s}, 0\leq s\leq1.$$
\end{lemma}

\begin{theorem}\label{t22}
(Scattering). If~$0<M(u)^{\frac{1-s_{c}}{s_{c}}}E(u)/M(Q)^{\frac{1-s_{c}}{s_{c}}}E(Q)<1$~
and the first case of ~Proposition~\ref{p21}~holds, then ~$u(t)$~scatters as ~$t\rightarrow +\infty$~or
~$t\rightarrow -\infty.$~That means there exist~$\phi_{\pm}\in H^{1}$~such that
\begin{equation}\label{2.7}
 \lim_{t\rightarrow\pm\infty}\|u(t)-e^{-it\Delta}\phi_{\pm}\|_{H^{1}}=0.
\end{equation}
Consequently, we have that
\begin{equation}\label{2.8}
 \lim_{t\rightarrow\pm\infty}\|u(t)\|_{L^{p+1}}=0
\end{equation}
and
\begin{equation}\label{2.9}
 \lim_{t\rightarrow\pm\infty} \eta(t)^{2}=\frac{1}{2\omega_{1}} \frac{E(u)M(u)^{\frac{1-s_{c}}{s_{c}}}}{E(Q)M(Q)^{\frac{1-s_{c}}{s_{c}}}}.
\end{equation}
\end{theorem}

\section{Virial Identity and Blow-Up Conditions}

In the sequel we focus on the second case of Proposition \ref{p21}. Using the classical virial identity, we first derive
the upper bound on the finite blow-up time under the finite variance hypothesis.

\begin{proposition}\label{p31}
Let $M(u)=M(Q),$~~$E(u)^{s_{c}}<E(Q)^{s_{c}}$. Suppose $\|xu_0\|_2<+\infty$ and suppose the second case of Proposition \ref{p21} holds
(~$\lambda>1$~ is defined by \eqref{2.6}).Define  $r(t)$ to be the scaled variance:
$$r(t)=\frac{\|xu\|_{2}^{2}}{\left(-16\omega_{1}\lambda^{2}+4N(p-1)\omega_{2}\lambda^{\frac{N(p-1)}{2}}\right)E(Q)}$$
 Then blow-up occurs in forward time before~$t_{b},$ where
 $$t_{b}=r'(0)+\sqrt{r'(0)^{2}+2r(0)}.$$
\end{proposition}

Note that
$$r(0)=\frac{\|xu_{0}\|_{2}^{2}}{\left(-16\omega_{1}\lambda^{2}+4N(p-1)\omega_{2}\lambda^{\frac{N(p-1)}{2}}\right)E(Q)}$$
and
$$r'(0)=\frac{Im\int(x\cdot\nabla u_{0})\overline{u_{0}}}{\left(-4\omega_{1}\lambda^{2}+N(p-1)\omega_{2}\lambda^{\frac{N(p-1)}{2}}\right)E(Q)}.$$

\begin{proof}
The virial identity gives
$$r''(t)=\frac{4N(p-1)E(u)-\left(2N(p-1)-8\right)\|\nabla u\|_{2}^{2}
}{\left(-16\omega_{1}\lambda^{2}+4N(p-1)\omega_{2}\lambda^{\frac{N(p-1)}{2}}\right)E(Q)}.$$
Identities  \eqref{2.2} imply
$$r''(t)=\frac{4N(p-1)\frac{E(u)}{E(Q)}-2\omega_{1}\left(2N(p-1)-8\right)\frac{\|\nabla u\|_{2}^{2}}{\|\nabla Q\|_{2}^{2}}
}{-16\omega_{1}\lambda^{2}+4N(p-1)\omega_{2}\lambda^{\frac{N(p-1)}{2}}}.$$
By the definition of~$\lambda$~and~$\eta,$~
$$r''(t)=\frac{4N(p-1)(\omega_{1}\lambda^{2}-\omega_{2}\lambda^{\frac{N(p-1)}{2}})-
2\omega_{1}\left(2N(p-1)-8\right)\eta(t)^{2}
}{-16\omega_{1}\lambda^{2}+4N(p-1)\omega_{2}\lambda^{\frac{N(p-1)}{2}}}.$$
Since $\eta(t)\geq \lambda>1,$ we have
$$r''(t)\leq-1,$$
which, by
integrating in time twice, gives
$$r(t)\leq-\frac{1}{2}t^{2}+r'(0)t+r(0).$$
The positive root of the polynomial on the right hand side is ~$t_{b}$~given in the proposition statement.

\end{proof}

The next result is related to the local virial identity. Let ~$\varphi\in C_{c}^{\infty}(\mathbb{R}^{N})$~be radial such that
\[
\varphi(x)=
\begin{cases}
|x|^{2}, &|x|\le 1;\\
0,  & |x|\geq 2.
\end{cases}
\]
For ~$R>0$~define
\begin{equation}\label{3.1}
z_{R}(t)=\int R^{2}\phi(\frac{x}{R})|u(x,t)|^{2} dx.
\end{equation}
Then we can directly calculate the following local virial identity:
\begin{align}\label{3.2}
z''_{R}(t)&=4\int \partial_{j}\partial_{k}\phi(\frac{x}{R})\partial_{j}u\partial_{k}\bar{u} dx
-\int\Delta\phi(\frac{x}{R})|u|^{4} dx-\frac{1}{R^{2}}\int\Delta^{2}\phi(\frac{x}{R})|u|^{2} dx\\
&=\left(4N(p-1)E(u)-\left(2N(p-1)-8\right)\|\nabla u\|_{2}^{2}\right)+A_{R}(u(t))\nonumber,
\end{align}
where for a constant~$C_{1}$~ we can control
\begin{align}\label{3.3}
A_{R}(u(t))\leq C_{1}\left( \frac{1}{R^{2}}\|u\|^{2}_{L^{2}(|x|\geq R)}+\|u\|^{p+1}_{L^{p+1}(|x|\geq R)}\right).
\end{align}

The local virial identity will give another version of Proposition~\ref{p31}~, for which ,  without the assumption
of finite variance, we will assumes that the solution is suitably localized in ~$H^{1}$~for all times. Define
$$\eta_{\geq R}=\frac{\|u\|^{s_{c}(p-1)}_{L^{2}(|x|\geq R)}\|\nabla u\|^{(1-s_{c})(p-1)}_{L^{2}(|x|\geq R)}
}{\|Q\|^{s_{c}(p-1)}_{2}\|\nabla Q\|^{(1-s_{c})(p-1)}_{2}}.$$

\begin{proposition}\label{p32}
Let~$M(u)=M(Q),$~~$E(u)<E(Q)$~and suppose the second case of Proposition~\ref{p21}~holds
(~$\lambda>1$~ is defined in~ \eqref{2.6}). Select~$\gamma$~such that
$$0<\gamma<min\left(2\omega_{1}\left(2N(p-1)-8\right),
4N(p-1)\omega_{2}\lambda^{\frac{N(p-1)-4}{2}}-16\omega_{1}\right).$$~
 Suppose that there is a radius~$R\geq C_{2}\gamma^{-\frac{1}{2}}$~such that for all ~$t,$~
there holds~$\eta_{\geq R}\leq \gamma.$~
Define~$\tilde{r}(t)$~to be the scaled local variance:
$$\tilde{r}(t)=
\frac{z_{R}(t)}{CE(Q)\left(-16\omega_{1}\lambda^{2}+4N(p-1)\omega_{2}\lambda^{\frac{N(p-1)}{2}}-
\gamma\lambda^{2}\right)}$$
( $C$ is an absolute constant determined by $C_1$ and $C_2$).
Then blow-up occurs in forward time before~$t_{b},$~where
 $$t_{b}=\tilde{r}'(0)+\sqrt{\tilde{r}'(0)^{2}+2\tilde{r}(0)}.$$

\end{proposition}

\begin{proof}
By the local virial identity and the same steps in the proof of Proposition~\ref{p31}~

$$\tilde{r}''(t)=\frac{1}{C}\frac{4N(p-1)(\omega_{1}\lambda^{2}-\omega_{2}\lambda^{\frac{N(p-1)}{2}})-
2\omega_{1}\left(2N(p-1)-8\right)\eta(t)^{2}+A_{R}(u(t))/E(Q)
}{-16\omega_{1}\lambda^{2}+4N(p-1)\omega_{2}\lambda^{\frac{N(p-1)}{2}}-
\gamma\lambda^{2}}$$
By the exterior Gagliardo-Nirenberg inequality, we have
\begin{align}\label{3.4}
\|u\|_{L^{p+1}(|x|\geq R)}^{p+1}\leq C_{GN}\|\nabla
u\|_{L^{2}(|x|\geq R)}^{\frac{N(p-1)}{2}}\|u\|_{L^{2}(|x|\geq R)}^{2-\frac{(N-2)(p-1)}{2}}
\leq\|\nabla u\|_{2}^{2}\eta_{\geq R}\leq\|\nabla Q\|_{2}^{2}\gamma \eta(t)^{2}.
\end{align}
This combined with
\begin{align}\label{3.5}
\frac{1}{R^{2}}\|u\|^{2}_{L^{2}(|x|\geq R)}\leq C_{2}^{-2}\|Q\|_{2}^{2}\gamma\leq C_{2}^{-2}\|Q\|_{2}^{2}\gamma \eta(t)^{2}
\end{align}
gives
\begin{align*}
\tilde{r}''(t)&\leq\frac{1}{C}\frac{4N(p-1)(\omega_{1}\lambda^{2}-\omega_{2}\lambda^{\frac{N(p-1)}{2}})-
2\omega_{1}\left(2N(p-1)-8\right)\eta(t)^{2}+C_{3}\frac{(\|Q\|_{2}^{2}+\|\nabla Q\|_{2}^{2})}{E(Q)}\gamma\eta(t)^{2}
}{-16\omega_{1}\lambda^{2}+4N(p-1)\omega_{2}\lambda^{\frac{N(p-1)}{2}}-
\gamma\lambda^{2}}\\
&\leq\frac{1}{C}\frac{C_{4}\left(4N(p-1)(\omega_{1}\lambda^{2}-\omega_{2}\lambda^{\frac{N(p-1)}{2}})-
2\omega_{1}\left(2N(p-1)-8\right)\eta(t)^{2}+\gamma\eta(t)^{2}\right)}{-16\omega_{1}\lambda^{2}+4N(p-1)\omega_{2}\lambda^{\frac{N(p-1)}{2}}-
\gamma\lambda^{2}}.
\end{align*}
Taking   the constant~$C=C_{4},$  since $\eta(t)\geq \lambda>1$ and from the selection of $\gamma,$ we obtain
$$r''(t)\leq-1.$$
The remainder of the argument is the same
 as in the proof of Proposition \ref{p31} .

\end{proof}

Finally, we will give the finite blow-up time for radial solutions before which we would like to introduce the Radial Gagliardo-Nirenberb inequality:
\begin{lemma}\label{raphael}\cite{raphael}
(Radial Gagliardo-Nirenberb inequality). For all~$\delta>0$,~there exists a constant $C_{\delta}>0$ such that for all $u\in \dot{H}^{s_{c}}$
with radial symmetry, and for all $R>0,$ we have
$$\int_{|x|\geq R}|u|^{p+1}dx\leq\delta\int_{|x|\geq R}|\nabla u|^{2}dx
+\frac{C_{\delta}}{R^{2(1-s_{c})}}\left[\left(\rho(u,R)\right)^{\frac{2(p+3)}{5-p}}+\left(\rho(u,R)\right)^{\frac{p+1}{2}}\right],$$
where ~$\rho(u,R)=\sup_{R'\geq R}\frac{1}{(R')^{2s_{c}}}\int_{R'\leq|x|\leq2R'}|u|^{2}dx.$~

\end{lemma}
Note that this lemma implies that
for all~$\delta>0$,~there exists a constant~$C_{\delta}>0$~and~$C_{Q}>0$~such that for all~$u\in \dot{H}^{s_{c}}$~
with radial symmetry and~$M(u)=M(Q)$~, and for all $R>0,$ we have
\begin{align}\label{radialsobolev}
\int_{|x|\geq R}|u|^{p+1}dx\leq\delta\int_{|x|\geq R}|\nabla u|^{2}dx
+\frac{C_{\delta}C_{Q}}{R^{2(1-s_{c})}}.
\end{align}

\begin{proposition}\label{p33}
Let~$M(u)=M(Q),$~~$E(u)<E(Q)$~and suppose the second case of Proposition~\ref{p21}~holds
(~$\lambda>1$~ is defined in~ \eqref{2.6}.) Suppose that ~$u$~is radial.
Select~$\gamma$~such that
$$0<\gamma<min\left(2\omega_{1}\left(2N(p-1)-8\right),
4N(p-1)\omega_{2}\lambda^{\frac{N(p-1)-4}{2}}-16\omega_{1}\right).$$~Then for
$$R>\max\left(\gamma^{-\frac{1}{2}},\left(\frac{2C_{\gamma}}{-16\omega_{1}\lambda^{2}+4N(p-1)\omega_{2}\lambda^{\frac{N(p-1)}{2}}-
\gamma\lambda^{2}}\right)^{\frac{1}{2(1-s_{c})}}\right)$$
we define~$\tilde{r}(t)$~to be the scaled local variance:
$$\tilde{r}(t)=
\frac{z_{R}(t)}{\tilde{C}_{Q}E(Q)\left(-16\omega_{1}\lambda^{2}+4N(p-1)\omega_{2}\lambda^{\frac{N(p-1)}{2}}-
\gamma\lambda^{2}\right)}$$
(where the constant  $\tilde{C}_{Q}$ is dependent on $Q$  determined by $C_\gamma$ and $C_Q$ in \eqref{radialsobolev}).
Then blow-up occurs in forward time before~$t_{b},$~where
 $$t_{b}=\tilde{r}'(0)+\sqrt{\tilde{r}'(0)^{2}+2\tilde{r}(0)}.$$

\end{proposition}

\begin{proof}
We modify the proof of Proposition \ref{p32} only in  \eqref{3.4} and \eqref{3.5}.
From the
Radial Gagliardo-Nirenberb inequality \eqref{radialsobolev} with $\delta=\gamma$, we obtain
\begin{align*}\label{}
\|u\|_{L^{p+1}(|x|\geq R)}^{p+1}\leq C_{Q}\left(\gamma\eta(t)^{2}+\frac{C_{\gamma}}{R^{2(1-s_{c})}}\right).
\end{align*}
If taking $C_Q$ to stand for the variant  constants dependent on $Q$, we have
\begin{align*}\label{}
\frac{1}{R^{2}}\|u\|^{2}_{L^{2}(|x|\geq R)}\leq \frac{C_{Q}}{R^{2}}\leq\frac{C_{Q}\eta(t)^{2}}{R^{2}}
\leq C_{Q}\gamma\eta(t)^{2}.
\end{align*}
Thus
\begin{align*}
\tilde{r}''(t)
&\leq C_{Q}\frac{4N(p-1)(\omega_{1}\lambda^{2}-\omega_{2}\lambda^{\frac{N(p-1)}{2}})-
2\omega_{1}\left(2N(p-1)-8\right)\eta(t)^{2}+\gamma\eta(t)^{2}+\frac{C_{\gamma}}{R^{2(1-s_{c})}}
}{\tilde{C}_{Q}\left(-16\omega_{1}\lambda^{2}+4N(p-1)\omega_{2}\lambda^{\frac{N(p-1)}{2}}-
\gamma\lambda^{2}\right)}\\
&\leq C_{Q}\frac{\left(4N(p-1)(\omega_{1}\lambda^{2}-\omega_{2}\lambda^{\frac{N(p-1)}{2}})-
2\omega_{1}\left(2N(p-1)-8\right)\eta(t)^{2}+\gamma\eta(t)^{2}\right)+\frac{C_{\gamma}}{R^{2(1-s_{c})}}
}{\tilde{C}_{Q}\left(-16\omega_{1}\lambda^{2}+4N(p-1)\omega_{2}\lambda^{\frac{N(p-1)}{2}}-
\gamma\lambda^{2}\right)}.
\end{align*}
We only have to select~$\tilde{C}_{Q}=2C_{Q}$~in the assumptions. Then since~$\eta(t)\geq\lambda>1,$~ the restriction of~$\gamma$~and~$R$~gives
$$r''(t)\leq-1,$$
and we conclude the proof with the same steps as in the proof of~Proposition \ref{p31}.

\end{proof}

\section{Variational Characterization of the Ground State}
This section deals with the variation characterization of~$Q$~stated in the above section.
It is an important preparation  for the  ``near boundary case" in Section~5. For now, we will write~$u=u(x)$~ as the
time dependence plays no role in what follows.

\begin{proposition}\label{p41}
There exists a function~$\epsilon(\rho)$~with~$\epsilon(\rho)\rightarrow 0$~as~$\rho\rightarrow 0$~
such that the following holds:
suppose there is~$\lambda>0$~satisfying
\begin{equation}\label{4.1}
\left|\frac{M(u)^{\frac{1-s_{c}}{s_{c}}}E(u)}{M(Q)^{\frac{1-s_{c}}{s_{c}}}E(Q)}
-\left(\omega_{1}\lambda^{2}-\omega_{2}\lambda^{\frac{N(p-1)}{2}}\right)\right|
\leq \rho \lambda^{\frac{N(p-1)}{2}},
\end{equation}
and
\begin{equation}\label{4.2}
\left|\frac{\|u\|_{2}^{\frac{1-s_{c}}{s_{c}}}\|\nabla u\|_{2}}{\|Q\|_{2}^{\frac{1-s_{c}}{s_{c}}}\|\nabla Q\|_{2}}
-\lambda\right|
\leq \rho
\begin{cases}\lambda, & \lambda\geq 1\\
\lambda^{2}, & \lambda\leq 1.
\end{cases}
\end{equation}
Then there exists~$\theta\in\mathbb{R}$~and~$x_{0}\in\mathbb{R}^{N}$~such that
\begin{equation}\label{4.3}
\left\|u-e^{i\theta}\lambda^{\frac{N}{2}}\beta^{-\frac{2}{p-1}}Q\left(\lambda(\beta^{-1}\cdot-x_{0})
\right)
\right\|_{2}\leq\beta^{\frac{N}{2}-\frac{2}{p-1}}\epsilon(\rho)
\end{equation}
and
\begin{equation}\label{4.4}
\left\|\nabla\left[u-e^{i\theta}\lambda^{\frac{N}{2}}\beta^{-\frac{2}{p-1}}Q\left(\lambda(\beta^{-1}\cdot-x_{0})
\right)\right]
\right\|_{2}\leq\lambda\beta^{\frac{N}{2}-\frac{2}{p-1}-1}\epsilon(\rho),
\end{equation}
where~$\beta=(\frac{M(u)}{M(Q)})^{\frac{p-1}{N(p-1)-4}}.$~

\end{proposition}
\begin{remark}\label{r43}
If we let~$v(x)=\beta^{\frac{2}{p-1}}u(\beta x),$~then~$M(v)=\beta^{\frac{4}{p-1}-N}M(u)=M(Q),$~
and we can then restate Proposition~\ref{p41}~as follows:

Suppose ~$\|v\|_{2}=\|Q\|_{2}$~and there is~$\lambda>0$~such that

\begin{equation}\label{4.5}
\left|\frac {E(v)}{E(Q)}
-\left(\omega_{1}\lambda^{2}-\omega_{2}\lambda^{\frac{N(p-1)}{2}}\right)\right|
\leq \rho \lambda^{\frac{N(p-1)}{2}},
\end{equation}
and
\begin{equation}\label{4.6}
\left|\frac{\|\nabla v\|_{2}}{\|\nabla Q\|_{2}}
-\lambda\right|
\leq \rho
\begin{cases}\lambda, & \lambda\geq 1\\
\lambda^{2}, & \lambda\leq 1.
\end{cases}
\end{equation}
Then there exists~$\theta\in\mathbb{R}$~and~$x_{0}\in\mathbb{R}^{N}$~such that
\begin{equation}\label{4.7}
\left\|v-e^{i\theta}\lambda^{\frac{N}{2}}Q\left(\lambda(\cdot-x_{0})
\right)
\right\|_{2}\leq \epsilon(\rho)
\end{equation}
and
\begin{equation}\label{4.8}
\left\|\nabla\left[v-e^{i\theta}\lambda^{\frac{N}{2}}Q\left(\lambda(\cdot-x_{0})
\right)\right]
\right\|_{2}\leq\lambda\epsilon(\rho).
\end{equation}

\end{remark}
Thus it suffices to prove the scaled statement equivalent to Proposition~\ref{p41}
and we will carry it out by means of the following
 result from Lions~\cite{lions}.

 \begin{proposition}\label{p44}
(\cite{lions}) There exists a function~$\epsilon(\rho),$~defined for small~$\rho>0$
~such that\\~$\lim_{\rho\rightarrow 0}\epsilon(\rho)=0,$~such that for all ~$u\in H^{1}$~with
\begin{equation}\label{4.9}
\left|\|u\|_{p+1}-\|Q\|_{p+1}\right|+\left|\|u\|_{2}-\|Q\|_{2}\right|+\left|\|\nabla u\|_{2}-\|\nabla Q\|_{2}\right|
\leq \rho,
\end{equation}
there exist~$\theta_{0}\in\mathbb{R}$~and~$x_{0}\in\mathbb{R}^{N}$~such that
\begin{equation}\label{4.10}
\left\|u-e^{i\theta_{0}}Q(\cdot-x_{0})\right\|_{H^{1}}
\leq \epsilon(\rho).
\end{equation}
\end{proposition}

\begin{proof}(Proof of proposition~\ref{p41}).
As a result of Remark~\ref{r43}, we will just prove the equivalent version rescaling off the mass.
Set~$\tilde{u}(x)=\lambda^{-\frac{N}{2}}v(\lambda^{-1}x),$~and then
\eqref{4.6}~gives
\begin{equation}\label{4.11}
\left|\frac{\|\nabla \tilde{u}\|_{2}}{\|\nabla Q\|_{2}}-1\right|
\leq \rho.
\end{equation}
On the other hand, by~\eqref{2.2}~and the notation of ~$\omega_{1}$~and ~$\omega_{2}$~we have
\begin{align*}
\left|\frac{\|v\|_{p+1}^{p+1}}{\|Q\|_{p+1}^{p+1}}-\lambda^{\frac{N(p-1)}{2}}\right|
&\leq\left|-\frac{1}{\omega_{2}}\left(\frac{E(v)}{E(Q)}-(\omega_{1}\lambda^{2}-\omega_{2}\lambda^{\frac{N(p-1)}{2}})\right)
\right|+\left|\frac{N(p-1)}{4}\frac{\|\nabla v\|_{2}^{2}}{\|\nabla Q\|_{2}^{2}}
-\frac{\omega_{1}}{\omega_{2}}\lambda^{2}\right|\\
&=\frac{1}{\omega_{2}}\left|\frac{E(v)}{E(Q)}-(\omega_{1}\lambda^{2}-\omega_{2}\lambda^{\frac{N(p-1)}{2}})
\right|+\frac{N(p-1)}{4}\left|\frac{\|\nabla v\|_{2}^{2}}{\|\nabla Q\|_{2}^{2}}
-\lambda^{2}\right|.
\end{align*}

Then~\eqref{4.5}~and~\eqref{4.6}~imply
\begin{align*}
\left|\frac{\|v\|_{p+1}^{p+1}}{\|Q\|_{p+1}^{p+1}}-\lambda^{\frac{N(p-1)}{2}}\right|
&\leq\frac{1}{\omega_{2}}\rho\lambda^{\frac{N(p-1)}{2}}+\frac{N(p-1)}{4}\rho
\begin{cases}\lambda^{2}, & \lambda\geq 1\\
\lambda^{4}, & \lambda\leq 1
\end{cases}\\
&\leq(\frac{N(p-1)}{2}-1)\rho\lambda^{\frac{N(p-1)}{2}}.
\end{align*}
Thus in terms of~$\tilde{u},$~we obtain
\begin{equation}\label{4.12}
\left|\frac{\| \tilde{u}\|_{p+1}^{p+1}}{\| Q\|_{p+1}^{p+1}}-1\right|
\leq \frac{N(p-1)-2}{2}\rho.
\end{equation}
Thus ~\eqref{4.11}~and~\eqref{4.12}~imply that the condition~\eqref{4.9}~is satisfied by~$\tilde{u}.$~
By Proposition~\ref{p44}~and rescaling back to ~$v,$~ we obtain~\eqref{4.7}~and~\eqref{4.8}~.

\end{proof}

\section{Near-Boundary Case}

We know from Proposition~\ref{p21}~that if~$M(u)=M(Q)$~and~$E(u)/E(Q)=
\omega_{1}\lambda^{2}-\omega_{2}\lambda^{\frac{N(p-1)}{2}}$~
for some~$\lambda>1$~and~$\|\nabla u_{0}\|_{2}/\|\nabla Q\|_{2}\geq\lambda,$~
then~$\|\nabla u(t)\|_{2}/\|\nabla Q\|_{2}\geq\lambda$~for all~$t.$~
Now in this section, we will claim that~$\|\nabla u(t)\|_{2}/\|\nabla Q\|_{2}$~cannot remain near~$\lambda$~
globally in time.

\begin{proposition}\label{p51}
Let~$\lambda_{0}>1.$~There exists~$\rho_{0}=\rho_{0}(\lambda_{0})>0$~with the property that~$\rho_{0}(\lambda_{0})\rightarrow 0$~
as~$\lambda_{0}\rightarrow 1,$~such that for any~$\lambda\geq\lambda_{0},$~the following holds:
There does not exist a solution~$u(t)$~of problem~\eqref{1.1}~with~$P(u)=0$~satisfying ~$M(u)=M(Q),$~
\begin{equation}\label{5.1}
\frac{E(u)}{E(Q)}=
\omega_{1}\lambda^{2}-\omega_{2}\lambda^{\frac{N(p-1)}{2}},
\end{equation}
and for all~$t\geq0$~
\begin{equation}\label{5.2}
\lambda\leq\frac{\|\nabla u(t)\|_{2}}{\|\nabla Q\|_{2}}\leq\lambda(1+\rho_{0}).
\end{equation}
\end{proposition}

We would like to give another equivalent statement implied by this assertion:
For any solution~$u(t)$~to ~\eqref{1.1}~with~$P(u)=0$~satisfying ~$M(u)=M(Q),$~
\begin{equation*}\label{5.1}
\frac{E(u)}{E(Q)}=
\omega_{1}\lambda^{2}-\omega_{2}\lambda^{\frac{N(p-1)}{2}},
\end{equation*}
and for all~$t\geq0$~
\begin{equation*}\label{5.2}
\lambda\leq\frac{\|\nabla u(t)\|_{2}}{\|\nabla Q\|_{2}},
\end{equation*}
there exist a time~$t_{0}\geq0$~such that
\begin{equation*}\label{5.2}
\frac{\|\nabla u(t_{0})\|_{2}}{\|\nabla Q\|_{2}}\geq\lambda(1+\rho_{0}).
\end{equation*}

Before proving Proposition~\ref{p51}~we will firstly give a useful lemma the proof of which will be found in~\cite{holmer10}.

\begin{lemma}\label{l52}
Suppose that ~$u(t)$~with~$P(u)=0$~solving~\eqref{1.1}~satisfies, for all~$t$~
\begin{equation}\label{5.3}
\left\|u(t)-e^{i\theta(t)}Q(\cdot-x(t))\right\|_{H^{1}}
\leq \epsilon
\end{equation}
for some continuous functions~$\theta(t)$~and~$x(t).$~Then
$$\frac{|x(t)|}{t}\leq C\epsilon^{2}\ \ \ as\ \ t\rightarrow+\infty.$$
\end{lemma}

\begin{proof}(Proof of proposition~\ref{p51}).
To the contrary, we suppose that there exists a solution~$u(t)$~satisfying
$M(u)=M(Q),$~~$E(u)/E(Q)=
\omega_{1}\lambda^{2}-\omega_{2}\lambda^{\frac{N(p-1)}{2}}$~and
\begin{equation}\label{5.4}
\lambda\leq\frac{\|\nabla u(t)\|_{2}}{\|\nabla Q\|_{2}}\leq\lambda(1+\rho_{0}).
\end{equation}
Since
~$\|\nabla u(t)\|_{2}^{2}\geq\lambda^{2}\|\nabla Q\|_{2}^{2}=2\omega_{1}\lambda^{2}E(Q),$~
we have
\begin{align*}
&4N(p-1)E(u)-\left(2N(p-1)-8\right)\|\nabla u\|_{2}^{2}\\
\leq &-4\left(N(p-1)\omega_{2}\lambda^{\frac{N(p-1)}{2}}-4\omega_{1}\lambda^{2}
\right)E(Q).
\end{align*}
By Proposition~\ref{p41}, there exist functions~$\theta(t)$~and~$x(t)$~such that for~$\rho=\rho_{0}$~
\begin{equation}\label{5.5}
\left\|u(t)-e^{i\theta(t)}\lambda^{\frac{N}{2}}Q\left(\lambda(\cdot-x(t))
\right)
\right\|_{2}\leq \epsilon(\rho)
\end{equation}
and
\begin{equation}\label{5.6}
\left\|\nabla\left[u(t)-e^{i\theta(t)}\lambda^{\frac{N}{2}}Q\left(\lambda(\cdot-x(t))
\right)\right]
\right\|_{2}\leq\lambda\epsilon(\rho).
\end{equation}
By the continuity of the~$u(t)$~flow, we may assume~$\theta(t)$~and~$x(t)$~are continuous.
Let$$R(T)=\max\left(\max_{0\leq t\leq T}|x(t)|,\log \epsilon(\rho)^{-1}\right).$$
For fixed~$T,$~take~$R=R(T)$~in the local virial identity \eqref{3.2}~.
Then owing to the exponential localization of~$Q(x)$~, ~\eqref{5.5}~and~\eqref{5.6}~imply that,
$$\left|A_{R}(u(t))\right|\leq\frac{C}{2}\lambda^{2}
\left(\epsilon(\rho)+e^{-R(T)}\right)^{2}\leq C\lambda^{2}\epsilon(\rho)^{2}.$$
Taking $\rho=\rho_{0}$ small enough to make~$ \epsilon(\rho)$~small such that for all~~$0\leq t\leq T,$~
$$z''_{R}(t)\leq -CE(Q)(\lambda^{\frac{N(p-1)}{2}}-\lambda^{2}),$$
and so
$$\frac{z_{R}(T)}{T^{2}}\leq\frac{z_{R}(0)}{T^{2}}+\frac{z'_{R}(0)}{T}-CE(Q)(\lambda^{\frac{N(p-1)}{2}}-\lambda^{2}).$$
By definition of~$z_{R}(t)$~we have
$$|z_{R}(0)|\leq CR^{2}\|u_{0}\|_{2}^{2}=C\|Q\|_{2}^{2}R^{2}$$and
$$|z'_{R}(0)|\leq CR\|u_{0}\|_{2}\|\nabla u_{0}\|_{2}\leq
C\|Q\|_{2}\|\nabla Q\|_{2}R(1+\rho_{0})\lambda.$$
Consequently,
$$\frac{z_{2R(T)}(T)}{T^2}\leq C\left(\frac{R(T)^{2}}{T^{2}}+\frac{\lambda R(T)}{T}\right)-CE(Q)(\lambda^{\frac{N(p-1)}{2}}-\lambda^{2}).$$
Taking ~$T$~sufficiently large ,Lemma~\ref{l52}~implies
$$0\leq \frac{z_{2R(T)}(T)}{T^2}\leq C\left(\lambda\epsilon(\rho)^{2}-(\lambda^{\frac{N(p-1)}{2}}-\lambda^{2})\right)<0$$
provided taking ~$\rho_{0}$~small enough .

Note that ~$\rho_{0}$~is independent of ~$T$~. We then get a contradiction.

\end{proof}

\section{Profile Decomposition}

In this section we make some extension of the cubic profile decomposition
\cite{holmer10}~to our general case, and we review some work done by the author in~\cite{yuan}.

First of all , we introduce some notations.
We say that~$(q,r)~$is~$\dot{H}^{s}(\mathbb{R}^{N})$~admissible and denote it by
~$(q,r)\in\Lambda_{s}$~if
$$\frac{2}{q}+\frac{N}{r}=\frac{N}{2}-s,\ \ \ \frac{2N}{N-2s}<r<\frac{2N}{N-2}$$
Correspondingly, we denote~$(q',r')$~the dual~$\dot{H}^{s}(\mathbb{R}^{N})$~admissible by~$(q',r')\in\Lambda'_{s}$~if
~$(q,r)\in\Lambda_{-s}$~with~$(q',r')$~is the H\"older~ dual to~$(q,r).$~
We also define the following Srichartz norm
$$\|u\|_{S(\dot{H}^{s})}=\sup_{(q,r)\in\Lambda_{s}}\|u\|_{L_t^qL_x^r}
 $$
and the dual Strichartz norm
$$\|u\|_{S'(\dot{H}^{-s})}=\inf_{(q',r')\in\Lambda'_{s}}\|u\|_{L_t^{q'}L_x^{r'}}=\inf_{(q,r)\in\Lambda_{-s}}\|u\|_{L_t^{q'}L_x^{r'}},$$
where~$(q',r')$~is the H\"older~ dual to~$(q,r).$~
Also as in ~\cite{holmer10}~,  the notation~$S(\dot{H}^{s};I)$~and~$S'(\dot{H}^{s};I)$~indicate a restriction
to a time subinterval~$I\subset(-\infty,+\infty).$~

\begin{remark}\label{rad}
By notation ~$\|\cdot\|_{S(\dot{H}^{s_{c}})}$~ in the sequel, we will in fact add the restriction~$q\geq r$~
 to the definition of  ~$(q,r)\in\Lambda_{s_{c}}$~without affecting the future arguments for our main results in this paper,
 which is needed in the proof of  Lemma \ref{l63} below.

\end{remark}

Now we first restate the linear profile decomposition  below  which was shown in ~\cite{yuan} .
\begin{lemma}\label{lpd}
(Profile expansion). Let $\phi_{n}(x)$ be an uniformly bounded
sequence in $H^{1}$,
then for each M there exists a subsequence of $\phi_{n}$, also
denoted by $\phi_{n}$, and (1) for each $1\leq j\leq M$, there
exists a (fixed in n) profile $\tilde{\psi}^{j}(x)$ in $H^1$,
 (2) for each $1\leq j\leq M$, there exists a sequence(in n)of time
shifts $t_{n}^{j}$, (3) for each $1\leq j \leq M$, there exists a
sequence (in n) of space shifts $x_{n}^{j}$, (4) there exists a
sequence (in n) of remainders $\tilde{W}_{n}^{M}(x)$ in $H^1$, such that
$$\phi_{n}(x)=\sum_{j=1}^{M}e^{-it_{n}^{j}\Delta}\tilde{\psi}^{j}(x-x_{n}^{j})+\tilde{W}_{n}^{M}(x),$$
The time and space sequences have a pairwise divergence property,
i.e., for $1\leq j\neq k\leq M$, we have
\begin{equation*}\label{4.1}
\lim_{n\rightarrow+\infty}(
|t_{n}^{j}-t_{n}^{k}|+|x_{n}^{j}-x_{n}^{k}|)=+\infty.
\end{equation*}
The remainder sequence has the following asymptotic smallness
property:
\begin{equation*}\label{4.2}
\lim_{M\rightarrow+\infty}[\lim_{n\rightarrow+\infty}\|e^{it\Delta}\tilde{W}_{n}^{M}\|_{S(\dot{H}^{s_{c}})}]=0.
\end{equation*}
For fixed M and any $0\leq s\leq1$, we have the asymptotic
Pythagorean expansion:
\begin{equation*}\label{4.3}
\|\phi_{n}\|_{\dot{H}^{s}}^{2}=\sum_{j=1}^{M}\|\tilde{\psi}^{j}\|_{\dot{H}^{s}}^{2}+\|\tilde{W}_{n}^{M}\|_{\dot{H}^{s}}^{2}+o_{n}(1).
\end{equation*}
\end{lemma}

\begin{remark}\label{rpd}
We omit the proof of Lemma~\ref{lpd},~but would like to point out some modification from the statement in~\cite{yuan}:
In the reference the author introduced a concept of k-point~$(\frac{1}{r},\frac{1}{q})$~for which one can also refer to~\cite{kato},
and gave the proof in terms of that conception.
In fact, it is easy to check that ~$(\frac{1}{r},\frac{1}{q})$~is a p-point with the same~$p$~in our equation~\eqref{1.1}~
, if and only if ~$(q,r)\in\Lambda_{s_{c}}.$~Moreover,  it is interesting to note that
if~$(q,r)\in\Lambda_{s_{c}}$~ then
~$(\frac{q}{p},\frac{r}{p})\in\Lambda'_{s_{c}}.$~Thus we have the following Strichartz estimate which was frequently used in \cite{yuan}:
\begin{align*}
\left\|i\int^{t}_{0}e^{i(t-t')\Delta}|u|^{p-1}|u|(x,t')dt'\right\|_{L_{t}^{q}L_{x}^{r}}
\leq C\left\||u|^{p-1}|u|\right\|_{L_{t}^{\frac{q}{p}}L_{x}^{\frac{r}{p}}}
\leq C\left\|u\right\|^{p}_{L_{t}^{q}L_{x}^{r}}.
\end{align*}
 Furthermore, the author in~\cite{yuan} gave another useful claim and we will restate the equivalent version as follows:
For any ~$(q,r)\in\Lambda_{s_{c}}$,~~there exists~$(q_{1},r_{1})\in\Lambda_{0}$~and
~$(q'_{2},r'_{2})\in\Lambda'_{0}$~such that$$
\begin{cases}
\frac{1}{q'_{2}}=\frac{1}{q_{1}}+\frac{p-1}{q}\\
\frac{1}{r'_{2}}=\frac{1}{r_{1}}+\frac{p-1}{r}.
\end{cases}$$
Applying the above observation, our proof of Lemma~\ref{lpd}~will be almost the same as that in~\cite{yuan},
and that is why we will omit it here.
\end{remark}

Similar to Keraani~\cite{keraani}~and~\cite{km},~we give the following definition of the nonlinear profile:
\begin{definition}
Let~$V$~be a solution to the linear Schr\"{o}dinger equation. We say~$U$~is the nonlinear profile associated to
~$(V,\{t_{n}\}),$~if ~$U$~is a solution to the Hartree equation~\eqref{1.1}~satisfying
$$\|(U-V)(-t_{n})\|_{H^1}\rightarrow0\ \ \  as\ \  n\rightarrow\infty.$$
\end{definition}

Note that, similar to the arguments  in \cite{km},  by the local theory and
the proof of the existence of  wave operators,
there always exist a nonlinear profile associated to a given~$(V,\{t_{n}\}).$
Thus
 for every~$j$,~there exists a solution $v^j$ to \eqref{1.1} associated to~$(\tilde{\psi}^{j},\{t_{n}^{j}\})$~
such that$$\|v^{j}(\cdot-x_{n}^{j},-t_{n}^{j})-e^{-it_{n}^{j}\Delta}\tilde{\psi}^{j}(\cdot-x_{n}^{j})\|_{H^{1}}\rightarrow0
\ \ \  as\ \  n\rightarrow\infty.$$

If we let $NLH(t)\psi$ denote the solution to \eqref{1.1} with initial data $\psi$,
 by shifting the linear profile $\tilde{\psi}^{j}$ when necessary, we may denote  $v^{j}(-t_{n}^{j})$
as  $NLH(-t_{n}^{j})\psi^j$ with some $\psi^{j}\in H^1$. Thus using
the same method of replacing linear flows by nonlinear flows
 as applied  in~\cite{radial}
to give the following proposition:

\begin{proposition}\label{p61}
Let ~$\phi_{n}(x)$ ~be an uniformly bounded
sequence in ~$H^{1}$~.
There exists a subsequence of ~$\phi_{n}$~, also
denoted by ~$\phi_{n}$~, profiles~
$\psi^{j}(x)$~ in ~$H^1$~, and parameters~$x_{n}^{j}$~,~$t_{n}^{j}$~
so that for each~$M$~,
\begin{equation}\label{pc}
\phi_{n}(x)=\sum_{j=1}^{M}NLS(-t_{n}^{j})\psi^{j}(x-x_{n}^{j})+W_{n}^{M}(x),
\end{equation}
 where as~$n\rightarrow \infty$~\\
$\bullet $~For each~$j$~, either~$t_{n}^{j}=0,$~~$t_{n}^{j}\rightarrow+\infty$~or~$t_{n}^{j}\rightarrow-\infty.$~\\
$\bullet $~If ~$t_{n}^{j}\rightarrow+\infty,$~ then~$\|NLS(-t)\psi^{j}\|_{S(\dot{H}^{s_{c}};[0,\infty))}<\infty$~
and if~$t_{n}^{j}\rightarrow-\infty,$~then~\\$\|NLS(-t)\psi^{j}\|_{S(\dot{H}^{s_{c}};[-\infty,0))}<\infty$~\\
 $\bullet $~For~$j\neq k$~,
 \begin{equation*}\label{4.1}
\lim_{n\rightarrow+\infty}(
|t_{n}^{j}-t_{n}^{k}|+|x_{n}^{j}-x_{n}^{k}|)=+\infty.
\end{equation*}
$\bullet $~~$NLS(t)W_{n}^{M}$~is global for~$M$~large enough with
 \begin{equation*}\label{4.2}
\lim_{M\rightarrow+\infty}[\lim_{n\rightarrow+\infty}\|NLS(t)W_{n}^{M}\|_{S(\dot{H}^{s_{c}})}]=0.
\end{equation*}
 We also have the ~$H^{s}$~Pythagorean decomposition: for fixed~$M$~and~$0\leq s\leq 1$~,
\begin{equation}\label{6.1}
\|\phi_{n}\|_{\dot{H}^{s}}^{2}=\sum_{j=1}^{M}\|NLS(-t_{n}^{j})\psi^{j}\|_{\dot{H}^{s}}^{2}+\|W_{n}^{M}\|_{\dot{H}^{s}}^{2}+o_{n}(1),
\end{equation}
and the energy ~Pythagorean decomposition
\begin{equation}\label{6.2}
E(\phi_{n})=\sum_{j=1}^{M}E(\psi^{j})+E(W_{n}^{M})+o_{n}(1).
\end{equation}

\end{proposition}
From a similar argument in~\cite{radial},~ we know that~ \eqref{6.2}~was proven by establishing the following first
\begin{equation}\label{6.3}
\|\phi_{n}\|_{p+1}^{p+1}=\sum_{j=1}^{M}\|NLS(-t_{n}^{j})\psi^{j}\|_{p+1}^{p+1}+\|W_{n}^{M}\|_{p+1}^{p+1}+o_{n}(1).
\end{equation}

The next lemma is an extension of the perturbation theory for the case~$N=3$~\cite{radial}.
By virtue of Remark~\ref{rpd},~the proof will also be similar to~\cite{yuan},~which we will represent in this paper.

\begin{lemma}\label{l62}
(Perturbation Theory). For each~ $A\geq 1$~, there exists~
$\epsilon_{0}=\epsilon_{0}(A)\ll1$~and~ $c=c(A)$ such that the following holds:
Fix ~$T>0$~. Let~$u=u(x,t)\in L^{\infty}([0,T];H^{1})$ ~solve
$$iu_{t}+\Delta
u+|u|^{p-1}u=0$$on~$[0,T]$~. Let
~ $\tilde{u}=\tilde{u}(x,t)\in L^{\infty}([0,T];H^{1})$~ and define
$$e =i\tilde{u}_{t}+\Delta
\tilde{u}+|\tilde{u}|^{p-1}\tilde{u}.$$
For each~$\epsilon\leq\epsilon_{0}$~, if for some~$(q_{1},r_{1})\in\Lambda_{-s_{c}}$~
$$\|\tilde{u}\|_{S(\dot{H}^{s_{c}};[0,T])}\leq A,\ \
\|e\|_{S'(\dot{H}^{-s_{c}};[0,T])}\leq \epsilon,\ \ and\ \
\|e^{it\Delta}(u(0)-\tilde{u}(0))\|_{S(\dot{H}^{s_{c}};[0,T])}\leq
\epsilon,$$
then
$$\|u-\tilde{u}\|_{S(\dot{H}^{s_{c}};[0,T])}\leq c(A) \epsilon. $$
\end{lemma}

\begin{proof}
Under the condition of the lemma, it suffices to prove that
for any~$(q,r)\in\Lambda_{s_{c}}$~and for some~$(q_{1},r_{1})\in\Lambda_{-s_{c}}$,~if
$$\|\tilde{u}\|_{L_{t}^{q}L_{x}^{r}}\leq
A,\ \
\|e\|_{L_{t}^{q'_1}L_{x}^{r'_1}}\leq \epsilon,\ \ and\ \
\|e^{it\Delta}(u(0)-\tilde{u}(0))\|_{L_{t}^{q}L_{x}^{r}}\leq
\epsilon,$$
then
$$\|u-\tilde{u}\|_{L_{t}^{q}L_{x}^{r}}\leq c(A) \epsilon. $$In fact,
the following arguments are similar to that  in \cite{yuan} except for some slight differences. One can also refer to \cite{radial} for
a similar proof.

Let $w$ defined by  $u=\tilde{u}+w$, then $w$ solves
\begin{equation}\label{a1}
iw_{t}+\Delta
w+|\tilde{u}+w|^{p-1}(\tilde{u}+w)-|\tilde{u}|^{p-1}\tilde{u}+e=0.
\end{equation}
Since $\|\tilde{u}\|_{L_{t}^{q}L_{x}^{r}}\leq A$, we can partition
$[0,T]$ into $N=N(A)$ intervals $I_{j}=[t_{j},t_{j+1}]$
such that for every $j$, $\|\tilde{u}\|_{L_{t\in
I_j}^{q}L_{x}^{r}}\leq\delta$ with $\delta$ sufficiently small to be
specified later.  The integral equation of \eqref{a1}
with initial data $w(t_{j})$ is
\begin{equation}\label{a2}
w(t)=e^{i(t-t_{j})\Delta}w(t_{j})+i\int_{t_{j}}^{t}e^{i(t-s)\Delta}W(\cdot,s)ds,
\end{equation}
where
$$W=(-|\tilde{u}+w|^{p-1}(\tilde{u}+w)+|\tilde{u}|^{p-1}\tilde{u})-e.$$
Applying
 the inhomogeneous Strichartz estimate in $I_{j}$ and from Remark \ref{rpd}, we have
\begin{eqnarray}
\|w\|_{L_{t\in I_j}^{q}L_{x}^{r}}&\leq&
\|e^{i(t-t_{j})\Delta}w(t_{j})\|_{L_{t\in I_j}^{q}L_{x}^{r}}\nonumber\\
&&+C||(-|\tilde{u}+w|^{p-1}(\tilde{u}+w)+|\tilde{u}|^{p-1}\tilde{u}||_{L_{t\in I_{j}}^{\frac{q}{p}}L_{x}^{\frac{r}{p}}} +\|e\|_{L_{t}^{q_1}L_{x}^{r_1}}\nonumber\\
&\leq&
\|e^{i(t-t_{j})\Delta}w(t_{j})\|_{L_{t\in I_j}^{q}L_{x}^{r}}\nonumber\\
&&+C||(|\tilde{u}|^{p-1}+|w|^{p-1})w||_{L_{t\in I_{j}}^{\frac{q}{p}}L_{x}^{\frac{r}{p}}} +\|e\|_{L_{t}^{q_1}L_{x}^{r_1}}\nonumber\\
&\leq&\|e^{i(t-t_{j})\Delta}w(t_{j})\|_{L_{t\in
I_j}^{q}L_{x}^{r}}\nonumber\\
&&+ C||\tilde{u}||^{p-1}_{L_{t\in
I_j}^{q}L_{x}^{r}}||w||_{L_{t\in
I_j}^{q}L_{x}^{r}}+C||w||_{L_{t\in
I_j}^{q}L_{x}^{r}}^{p}+\|e\|_{L_{t}^{q_1}L_{x}^{r_1}},\nonumber\\
&\leq&
\|e^{i(t-t_{j})\Delta}w(t_{j})\|_{L_{t\in I_j}^{q}L_{x}^{r}}\nonumber\\
&&+C\delta^{p-1}\|w\|_{L_{t\in I_j}^{q}L_{x}^{r}}+C\|w\|_{L_{t\in
I_j}^{q}L_{x}^{r}}^{p}+C\epsilon.\nonumber
\end{eqnarray}
If
\begin{equation}\label{a3}
\delta\leq (\frac{1}{4C})^{\frac{1}{p-1}},  \
(\|e^{i(t-t_{j})\Delta}w(t_{j})\|_{L_{t\in
I_j}^{q}L_{x}^{r}}+C\epsilon_{0})\leq
\frac{1}{2}(\frac{1}{4C})^{\frac{1}{p-1}},
\end{equation}
then
$$\|w\|_{L_{t\in I_j}^{q}L_{x}^{r}}\leq
2\|e^{i(t-t_{j})\Delta}w(t_{j})\|_{L_{t\in
I_j}^{q}L_{x}^{r}}+2C\epsilon.$$
 Now we take $t=t_{j+1}$ in
\eqref{a2}, and apply $e^{i(t-t_{j})\Delta}$ to the both sides
, we obtain
$$e^{i(t-t_{j+1})\Delta}w(t_{j+1})=e^{i(t-t_{j})\Delta}w(t_{j})+ i\int_{t_{j}}^{t_{j+1}}e^{i(t-s)\Delta}W(\cdot,s)ds. $$
Again, with the same method as above, we obtain
$$\|e^{i(t-t_{j+1})\Delta}w(t_{j+1})\|_{L_{t\in I_j}^{q}L_{x}^{r}}\leq
2\|e^{i(t-t_{j})\Delta}w(t_{j})\|_{L_{t\in
I_j}^{q}L_{x}^{r}}+2C\epsilon.$$
Iterating the above procedure
from $j=0$, we have
$$\|e^{i(t-t_{j})\Delta}w(t_{j})\|_{L_t^{q}L_{x}^{r}}\leq
2^{j}\|e^{i(t-t_{0})\Delta}w(t_{0})\|_{L_t^{q}L_{x}^{r}}+(2^{j}-1)2C\epsilon\leq
2^{j+2}C\epsilon.$$
To accommodate the second part of \eqref{a3} for all
intervals $I_{j},\ 0\leq j\leq N-1$, we require that
\begin{equation}\label{a4}
2^{N+2}C\epsilon_{0}\leq (\frac{1}{4C})^{\frac{1}{p-1}},
\end{equation}
and we obtain the result easily.

Now we recall
the parameter dependence of parameters:
We choose $\delta$ to meet the first
part of \eqref{a3}. Given A, the number of the interval $N$ is
determined, and the inequality \eqref{a4} tells how small
 $\epsilon_{0}$ should be
taken in terms of $N(A)$.

\end{proof}

Note from the proof above that, the parameters in Lemma \ref{l62} is not dependent on~$T$~.  As is stated in~\cite{holmer10}~for~$N=3$~, besides the~$H^{1}$~asymptotic orthogonality~\eqref{6.1}~ at~$t=0$~, this property
can be extended  to the NLS flow for~$0\leq t\leq T$~as an application of Lemma \ref{l62}  with an constant ~$A=A(T)$~ dependent
on~$T$ (but only through  $A$). As for the general Mass-supercritical
and Energy-subcritical case, we can prove  the following similar result:

\begin{lemma}\label{l63}
(~$H^{1}$~Pythagorean Decomposition Along the NLS Flow). Suppose ~$\phi_{n}(x)$ ~be a uniformly bounded
sequence in ~$H^{1}$~. Fix any time~$0<T<\infty$~. Suppose that~$u_{n}(t)\equiv NLS(t)\phi_{n}$~exists up
to time~$T$~for all~$n$~and
$$\lim_{n\rightarrow\infty}\|\nabla u_{n}(t)\|_{ L^{\infty}([0,T];L^{2})}<\infty.$$
Let ~$W_{n}^{M}(t)\equiv NLS(t)W_{n}^{M}.$~Then, for all~$j$~,
 ~$v^{j}(t)\equiv NLS(t)\psi^{j}$~exist up to time~$T$~and for all~$t\in[0,T],$~
\begin{equation}\label{6.4}
\|\nabla u_{n}\|_{2}^{2}=\sum_{j=1}^{M}\|\nabla v^{j}(t-t_{n}^{j})\|_{2}^{2}+\|\nabla W_{n}^{M}(t)\|_{2}^{2}+o_{n}(1).
\end{equation}
Here, ~$o_{n}(1)\rightarrow 0$~uniformly on~$0\leq t\leq T$~.

\end{lemma}

\begin{proof}
Let~$M_{0}$~be such that for~$M_{1}\geq M_{0}$~and for~$\delta_{sd}$~in Lemma~\ref{sd},
we have~$$\|NLS(t)W_{n}^{M_{1}}\|_{S(\dot{H}^{s_{c}})}\leq \delta_{sd}/2$$
and $\|v^{j}\|_{S(\dot{H}^{s_{c}})}\leq \delta_{sd}$ for $j>M_{0}.$ Reorder the first~$M_{0}$~
profiles and introduce an index~$M_{2}$~,~$0\leq M_{2}\leq M_{0}$~, such that\\
$\bullet$~For each~$0\leq j\leq M_{2}$~we have~$t_{n}^{j}=0$~.(There is no ~$j$~in this category if~$M_{2}=0.$~ )\\
$\bullet$~For each~$M_{2}+1\leq j\leq M_{0}$~we have~$|t_{n}^{j}|\rightarrow\infty.$~(There is no ~$j$~in this category if~$M_{2}=M_{0}.$~ )

By definition of $M_{0}$,
 $v^{j}(t)$ for $j>M_{0}$ scatters in both time directions.
We claim that for fixed ~$T$~and~$M_{2}+1\leq j\leq M_{0}$~, ~$\| v^{j}(t-t_{n}^{j})\|_{S(\dot{H}^{s_{c}};[0,T])}\rightarrow 0$~
as~$n\rightarrow\infty$~. Indeed, take the case ~$t_{n}^{j}\rightarrow+\infty$~for example. By Proposition~\ref{p61},
~$\| v^{j}(-t)\|_{S(\dot{H}^{s_{c}};[0,\infty))}<\infty$~. Then for~$q<\infty$~, ~$\| v^{j}(-t)\|_{L^{q}([0,\infty);L^{r})}<\infty$~
implies ~$\| v^{j}(t-t_{n}^{j})\|_{L^{q}([0,T];L^{r})}\rightarrow 0.$~
On the other hand, since ~$v^{j}(t)$~in Proposition~\ref{p61}~is constructed by the existence of wave operators which converge in
~$H^{1}$~to a linear flow at ~$-\infty$~, then the ~$L^{\frac{2N}{N-2s_{c}}}$~ decay of the linear flow implies immediately that
 ~$\| v^{j}(t-t_{n}^{j})\|_{L^{\infty}([0,T];L^{\frac{2N}{N-2s_{c}}})}\rightarrow 0.$~

 Let~$B=\max(1,\lim_{n}\|\nabla u_{n}\|_{L^{\infty}([0,T];L^{2})})$~. For each~$1\leq j\leq M_{2}$~, define~$T^{j}\leq T$~to be
 the maximal forward time on which ~$\|\nabla v^{j}\|_{L^{\infty}([0,T^{j}];L^{2})}\leq2B.$
 ~Let~$\widetilde{T}=\min_{1\leq j\leq M_{2}}T^{j}$~, and if $M_{2}=0,$ we just take $\widetilde{T}=T.$
 Note that if we have proved ~\eqref{6.4}~holds for ~$T=\widetilde{T}$~, then by definition of~$T^{j}$~, using the continuity arguments,
 it follows from~\eqref{6.4}~ that for each~$1\leq j\leq M_{2},$~we have~$T^{j}=T.$~Hence~$\widetilde{T}=T.$~
 Thus, for the remainder of the proof, we just work on~$[0,\widetilde{T}].$~

 For each~$1\leq j\leq M_{2}$~,
~$\| v^{j}\|_{L^{\infty}([0,\widetilde{T}];L^{2})}
 =\|\psi^{j}\|_{2}\leq \lim_{n}\|\phi_{n}\|_{2}$~by~\eqref{6.1}~. Now, in view of the notation of~$S(\dot{H}^{s_{c}};[0,\widetilde{T}])$~and
  Remark \ref{rad},  we will give the~$S(\dot{H}^{s_{c}};[0,\widetilde{T}])$-norm
 boundedness of~$v^{j}$~in two cases:

 Let ~$(\tilde{q},\tilde{r})=(\frac{(p-1)(N+2)}{2},\frac{(p-1)(N+2)}{2}).$~
 Case 1, if~$\tilde{r}\geq\frac{2N}{N-2}$~and thus~$(\frac{2}{1-s_{c}},\frac{2N}{N-2})\in\Lambda_{s_{c}},$~
 then
\begin{align*}
\| v^{j}(t)\|_{S(\dot{H}^{s_{c}};[0,\widetilde{T}])}
&\leq C(\| v^{j}\|_{L^{\infty}([0,\widetilde{T}];L^{\frac{2N}{N-2s_{c}}})}
+\| v^{j}\|_{L^{\frac{2}{1-s_{c}}}([0,\widetilde{T}];L^{\frac{2N}{N-2}})})\\
&\leq C(\| v^{j}\|_{L^{\infty}([0,\widetilde{T}];L^{2})}^{1-s_{c}}
\| \nabla v^{j}\|_{L^{\infty}([0,\widetilde{T}];L^{2})}^{s_{c}}
+\widetilde{T}^{\frac{1-s_{c}}{2}}\| \nabla v^{j}\|_{L^{\infty}([0,\widetilde{T}];L^{2})})\\
&\leq C(1+\widetilde{T}^{\frac{1-s_{c}}{2}})B.
\end{align*}
 Case 2, if on the other hand~$\tilde{r}<\frac{2N}{N-2}.$ Since clearly $(\tilde{q},\tilde{r})\in\Lambda_{s_{c}},$
 \begin{align*}
\| v^{j}(t)\|_{S(\dot{H}^{s_{c}};[0,\widetilde{T}])}
&\leq C(\| v^{j}\|_{L^{\infty}([0,\widetilde{T}];L^{\frac{2N}{N-2s_{c}}})}
+\| v^{j}\|_{L^{\frac{(p-1)(N+2)}{2}}([0,\widetilde{T}];L^{\frac{(p-1)(N+2)}{2}})})\\
&\leq C(\| v^{j}\|_{L^{\infty}([0,\widetilde{T}];L^{2})}^{1-s_{c}}
\| \nabla v^{j}\|_{L^{\infty}([0,\widetilde{T}];L^{2})}^{s_{c}}
+\widetilde{T}^{\frac{(p-1)(N+2)}{2}}\| \nabla v^{j}\|_{L^{\infty}([0,\widetilde{T}];L^{2})})\\
&\leq C(1+\widetilde{T}^{\frac{(p-1)(N+2)}{2}})B.
\end{align*}

For fixed ~$M$~, let$$\tilde{u}_{n}(x,t)=\sum_{j=1}^{M}v^{j}(x-x_{n}^{j},t-t_{n}^{j}),$$
and let$$e_{n}=i\partial_{t}\tilde{u}_{n}+\Delta\tilde{u}_{n}+|\tilde{u}_{n}|^{p-1}\tilde{u}_{n}.$$
We claim that there exists~$A=A(\widetilde{T})$~(independent of ~$M$~)such that for all~$M> M_{0},$~
there exists~$n_{0}=n_{0}(M)$~such that for all~$n> n_{0},$~
$$\|\tilde{u}_{n}\|_{S(\dot{H}^{s_{c}};[0,\widetilde{T}])}\leq A.
\footnote{We in fact prove both~$\|\tilde{u}_{n}\|_{L^{\frac{(p-1)(N+2)}{2}}([0,\widetilde{T}];L^{\frac{(p-1)(N+2)}{2}})}$~
and~$\|\tilde{u}_{n}\|_{L^{\infty}([0,\widetilde{T}];L^{\frac{(2N)(N-2s_{c})}{2}})}$~are bounded, and thus,
by interpolation, for any~$(q,r)\in\Lambda_{s_{c}}~(q\geq r),$~we obtain the
$\|\tilde{u}_{n}\|_{L^{q}([0,\widetilde{T}];L^{r})}$~bound.}
 $$
Furthermore, we also claim that for each~$M>M_{0}$~and~$\epsilon>0,$~there exists~$n_{1}=n_{1}(M,\epsilon)$~
such that for~$n> n_{1}$~and for some~$ (q,r)\dot{H}^{-s_{c}}$~admissible,
$$\|e_{n}\|_{L^{q'}([0,\widetilde{T}];L^{r'})}\leq \epsilon.$$
Both of the two claims have exactly been verified in ~\cite{yuan}(in the proof of Proposition 4.4 there) , we shall not prove them here again.
Moreover, since~$u_{n}(0)-\tilde{u}_{n}(0)=W_{n}^{M},$~there exists~$M'=M'(\epsilon)$~large enough such that
for each~$M>M'$~there exists~$n_{2}=n_{2}(M')$~such that for~$n> n_{2},$~
~$$\|e^{it\Delta}(u(0)-\tilde{u}(0))\|_{S(\dot{H}^{s_{c}};[0,\widetilde{T}])}\leq\epsilon.$$~
For ~$A=A(\widetilde{T})$~in the first claim, Lemma~\ref{l62}~gives us ~$\epsilon_{0}=\epsilon_{0}(A)\ll1.$
~We select an arbitrary~~$\epsilon\leq\epsilon_{0}$~and obtain from above arguments an index~$M'=M'(\epsilon).$~
Now select an arbitrary~$M>M'$~, and set~$n'=\max(n_{0},n_{1},n_{2}).$~Then  by Lemma~\ref{l62}~and the above
arguments, for~$n> n',$~we have
\begin{equation}\label{6.5}
\|u_{n}-\tilde{u_{n}}\|_{S(\dot{H}^{s_{c}};[0,T])}\leq c(\widetilde{T}) \epsilon.
\end{equation}
In order to obtain the~$\|\nabla\tilde{u}_{n}\|_{L^{\infty}([0,\widetilde{T}];L^{2})}$~bound,
we also have to discuss~$j\geq M_{2}+1.$~ As is noted in the first paragraph of the proof,
~$\| v^{j}(t-t_{n}^{j})\|_{S(\dot{H}^{s_{c}};[0,\widetilde{T}])}\rightarrow 0$~as~$n\rightarrow\infty.$~
By Strichartz estimate we can easily get~
$\|\nabla v^{j}(t-t_{n}^{j})\|_{L^{\infty}([0,\widetilde{T}];L^{2})}\leq C\|\nabla v^{j}(-t_{n}^{j})\|_{2}.$~
By the pairwise divergence of parameters,
\begin{align*}
\|\nabla\tilde{u}_{n}\|^{2}_{L^{\infty}([0,\widetilde{T}];L^{2})}
&=\sum_{j=1}^{M_{2}}\|\nabla v^{j}(t)\|^{2}_{L^{\infty}([0,\widetilde{T}];L^{2})}
+\sum_{M_{2}+1}^{M}\|\nabla v^{j}(t-t_{n}^{j})\|^{2}_{L^{\infty}([0,\widetilde{T}];L^{2})}+o_{n}(1)\\
&\leq C\left(M_{2}B^{2}+\sum_{M_{2}+1}^{M}\|\nabla NLS(-t_{n}^{j})\psi^{j}\|_{2}^{2}+o_{n}(1)\right)\\
&\leq C\left(M_{2}B^{2}+\|\nabla\phi_{n}\|_{2}^{2}+o_{n}(1)\right)\\
&\leq C\left(M_{2}B^{2}+B^{2}+o_{n}(1)\right).
\end{align*}
Note that ~$\frac{2N}{N-2s_{c}}<p+1<\frac{2N}{N-2},$~then for some ~$0<\theta<1$~and from~\eqref{6.5}~we have
\begin{align}\label{p+1}
\|u_{n}-\tilde{u}_{n}\|_{L^{\infty}([0,\widetilde{T}];L^{p+1})}
&\leq C\left(\|u_{n}-\tilde{u}_{n}\|^{\theta}_{L^{\infty}([0,\widetilde{T}];L^{\frac{2N}{N-2s_{c}}})}
\|\nabla(u_{n}-\tilde{u}_{n})\|^{1-\theta}_{L^{\infty}([0,\widetilde{T}];L^{2})}
\right)\\ \nonumber
&\leq c(\widetilde{T})^{\theta}\left(M_{2}B^{2}+B^{2}+o_{n}(1)\right)^{\frac{1-\theta}{2}}\epsilon^{\theta}.
\end{align}

Now in the sequel we first replace the large parameter
 ~$M$~ in the notation ~$\tilde{u}_{n}$~and all other arguments above for~$M_{1}$~ which appears at the beginning of our proof.
 Then for any fixed ~$M,$~we will prove  \eqref{6.4}  on $[0,\widetilde{T}].$  In fact, we need only to establish that, for each
 ~$t\in[0,\widetilde{T}],$
\begin{equation}\label{6.6}
\|u_{n}\|_{p+1}^{p+1}=\sum_{j=1}^{M}\| v^{j}(t-t_{n}^{j})\|_{p+1}^{p+1}+\| W_{n}^{M}(t)\|_{p+1}^{p+1}+o_{n}(1).
\end{equation}
Since then by~\eqref{6.2}~and the energy conservation we have
\begin{equation}\label{6.7}
E(u_{n}(t))=\sum_{j=1}^{M}E(v^{j}(t-t_{n}^{j}))+E(W_{n}^{M}(t))+o_{n}(1).
\end{equation}
Thus ~\eqref{6.6}~combined with~\eqref{6.7}~gives~\eqref{6.4}~, which completes our proof. So now what is the remainder
 is to establish~\eqref{6.6}.

 We first apply the perturbation theory Lemma \ref{l62} to $u_n=W_{n}^{M}$ and $\tilde{u}_n=\sum_{j=M+1}^{M_{1}}v^{j}(t-t_{n}^{j})$.
 For any fixed  $M<M_{1}$,
by the profile composition~\eqref{pc}~and the definition of~$W_{n}^{M}(t)$~and
 ~$v^{j}(t),$ similar to the above two claims and the arguments followed, we obtain
$$
\|W_{n}^{M}(t)-\sum_{j=M+1}^{M_{1}}v^{j}(t-t_{n}^{j})\|_{p+1}\rightarrow 0\ \ as\ \ n\rightarrow\infty.$$
Then by the pairwise divergence of parameters,
\begin{align*}\label{}
\| u_{n}\|_{p+1}^{p+1}&=\| \tilde{u}_{n}\|_{p+1}^{p+1}+o_{n}(1)\\
&=\|\sum_{j=1}^{M_{1}} v^{j}(t-t_{n}^{j})\|_{p+1}^{p+1}+o_{n}(1)\\
&=\sum_{j=1}^{M}\| v^{j}(t-t_{n}^{j})\|_{p+1}^{p+1}+\|\sum_{j=M+1}^{M_{1}}v^{j}(t-t_{n}^{j})\|_{p+1}^{p+1}+o_{n}(1)\\
&=\sum_{j=1}^{M}\| v^{j}(t-t_{n}^{j})\|_{p+1}^{p+1}+\| W_{n}^{M}(t)\|_{p+1}^{p+1}+o_{n}(1).
\end{align*}
If on the other hand~$M\geq M_{1}$,~ we then easily get from the selection of~$M_{1}$~at the beginning of our proof that
~$\|W_{n}^{M}(t)\|_{p+1}=o_{n}(1)$~and~\eqref{p+1}~implies ~\eqref{6.6}.

\end{proof}

\begin{lemma}\label{l64}
(Profile Reordering).Let ~$\phi_{n}(x)$ ~be a  bounded
sequence in ~$H^{1}$~and let ~$\lambda_{0}>1.$~ Suppose that ~$M(\phi_{n})=M(Q),$~
 ~$E(\phi_{n})/E(Q)=\omega_{1}\lambda_{n}^{2}-\omega_{2}\lambda_{n}^{\frac{N(p-1)}{2}}$~
with~$\lambda_{n}\geq\lambda_{0}>1$~and~$\|\nabla \phi_{n}\|_{2}/\|\nabla Q\|_{2}\geq\lambda_{n}$~for each~$n.$~
Then, for a given~$M,$~ the profiles can be reordered so that there exist~$1\leq M_{1}\leq M_{2}\leq M$~and\\
(1) ~For each~$1\leq j\leq M_{1},$~we have~$t_{n}^{j}=0$~and~$v^{j}(t)\equiv NLS(t)\psi^{j}$~does not scatter as
~$t\rightarrow+\infty.$~(We in fact assert that at least one~$j$~belongs to this category.)\\
(2)~For each~$M_{1}+1\leq j\leq M_{2},$~we have~$t_{n}^{j}=0$~and~$v^{j}(t)$~scatters as
~$t\rightarrow+\infty.$~(There is no  ~$j$~in this category if~$M_{2}=M_{1}.$~ )\\
(3)~For each~$M_{2}+1\leq j\leq M$~we have~$|t_{n}^{j}|\rightarrow\infty.$~(There is no ~$j$~in this category if~$M_{2}=M.$~ )

\end{lemma}

\begin{proof}
Firstly, we prove that there exists at least one~$j$~such that~$t_{n}^{j}$~converges as~$n\rightarrow\infty.$~ In fact,
\begin{align}\label{6.8}
\frac{\|\phi_{n}\|^{p+1}_{p+1}}{\| Q\|^{p+1}_{p+1}}
&=-\frac{1}{\omega_{2}}\frac{E(\phi_{n})}{E(Q)}+\frac{N(p-1)}{4}\frac{\|\nabla \phi_{n}\|_{2}^{2}}{\|\nabla Q\|_{2}^{2}}\\ \nonumber
&\geq-\frac{1}{\omega_{2}}\left(\omega_{1}\lambda_{n}^{2}-\omega_{2}\lambda_{n}^{\frac{N(p-1)}{2}}\right)
+\frac{N(p-1)}{4}\lambda_{n}^{2}\\ \nonumber
&=\lambda_{n}^{\frac{N(p-1)}{2}}\geq\lambda_{0}^{\frac{N(p-1)}{2}}>1.
\end{align}
If~$|t_{n}^{j}|\rightarrow\infty,$~then~$\|NLS(-t_{n}^{j})\psi^{j}\|_{p+1}\rightarrow 0$~and ~\eqref{6.3}~implies our claim.
Now if~$j$~is such that~$t_{n}^{j}$~converges as~$n\rightarrow\infty,$~ then we might as well assume ~$t_{n}^{j}=0.$~

Reorder the profiles~$\psi^{j}$~so that for~$1\leq j\leq M_{2},$~~we have~$t_{n}^{j}=0,$~
and for~$M_{2}+1\leq j\leq M$~we have~$|t_{n}^{j}|\rightarrow\infty.$~ What is the remainder is to show that there exists
one~$j,$~$1\leq j\leq M_{2},$~such that~$v^{j}(t)$~does not scatter as
~$t\rightarrow+\infty.$~To the contrary, if for all~$1\leq j\leq M_{2},$ $v^{j}(t)$ scatters,
 then we have $\lim_{t\rightarrow+\infty}\|v^{j}(t)\|_{p+1}=0.$ Let  $t_{0}$~be sufficiently large so that for all~
$1\leq j\leq M_{2},$~we have~$\|v^{j}(t_{0})\|_{p+1}^{p+1}\leq\epsilon/M_{2}.$~The ~$L^{p+1}$~orthogonality~\eqref{6.6}~
along the NLS flow and an argument as~\eqref{6.8}~ imply
\begin{align*}\label{}
\lambda_{0}^{\frac{N(p-1)}{2}}\| Q\|_{p+1}^{p+1}
&\leq\|u_{n}(t_{0})\|_{p+1}^{p+1}\\
&=\sum_{j=1}^{M_{2}}\| v^{j}(t_{0})\|_{p+1}^{p+1}+\sum_{j=M_{2}+1}^{M}\|v^{j}(t_{0}-t_{n}^{j})\|_{p+1}^{p+1}
+\| W_{n}^{M}(t_{0})\|_{p+1}^{p+1}+o_{n}(1).
\end{align*}
We know from ~Proposition~\ref{p61}~that,as~$n\rightarrow+\infty,$~~$\sum_{j=M_{2}+1}^{M}\|v^{j}(t_{0}-t_{n}^{j})\|_{p+1}^{p+1}\rightarrow 0,$~
and thus we have
\begin{align*}\label{}
\lambda_{0}^{\frac{N(p-1)}{2}}\| Q\|_{p+1}^{p+1}
\leq\epsilon
+\| W_{n}^{M}(t_{0})\|_{p+1}^{p+1}+o_{n}(1).
\end{align*}
This gives a contradiction since $W_n^M(t)$ is a scattering solution.

\end{proof}

\section{ Inductive Argument and Existence of a Critical Solution}

We now begin to prove  Theorem~\ref{th1}.~Note from Remark~\ref{p}~ that we have reduced Theorem~\ref{th1}~
to the case~$P(u)=0,$~thus we first
 give some definitions :
\begin{definition}\label{d71}
Let ~$\lambda>1.$~We say that~$\exists GB(\lambda,\sigma)$~holds if there exists a solution~$u(t)$~to\\~\eqref{1.1}~
such that$$P(u)=0,\ \ \ \ M(u)=M(Q),\ \ \ \frac{E(u)}{E(Q)}=\omega_{1}\lambda^{2}-\omega_{2}\lambda^{\frac{N(p-1)}{2}}$$
and$$\lambda\leq\frac{\|\nabla u(t)\|_{2}}{\|\nabla Q\|_{2}}\leq\sigma\ \ \ for\  all \ \ t\geq 0.$$
\end{definition}

$\exists GB(\lambda,\sigma)$~means that there exist solutions with energy~$\omega_{1}\lambda^{2}-\omega_{2}\lambda^{\frac{N(p-1)}{2}}$~globally bounded by~$\sigma.$~
Thus by Proposition~\ref{p51}~, ~$\exists GB(\lambda,\lambda(1+\rho_{0}(\lambda_{0})))$~is false for all~$\lambda\geq\lambda_{0}>1.$~

The statement ~$\exists GB(\lambda,\sigma)$~is false is equivalent to say that
for every solution~$u(t)$~to~\eqref{1.1}~\\with~$M(u)=M(Q)$~and~
$E(u)/E(Q)=\omega_{1}\lambda^{2}-\omega_{2}\lambda^{\frac{N(p-1)}{2}}$~
such that~$\|\nabla u(t)\|_{2}/\|\nabla Q\|_{2}\geq\lambda$~for all~$t,$~
there must exists a time ~$t_{0}\geq0$~such that~$\|\nabla u(t_{0})\|_{2}/\|\nabla Q\|_{2}\geq\sigma.$~
By resetting the initial time,  we can find a sequence~$t_{n}\rightarrow\infty$~such that
~$\|\nabla u(t_{n})\|_{2}/\|\nabla Q\|_{2}\geq\sigma$~for all~$n.$~

Note that if~$\lambda\leq\sigma_{1}\leq\sigma_{2},$~then ~$\exists GB(\lambda,\sigma_{2})$~is false implies
~$\exists GB(\lambda,\sigma_{1})$~is false. We will induct on the statement and define a threshold.

\begin{definition}\label{d72}
(The Critical Threshold.)~Fix~$\lambda_{0}>1.$~Let~$\sigma_{c}=\sigma_{c}(\lambda_{0})$~be the supremum of all
~$\sigma>\lambda_{0}$~such that~$\exists GB(\lambda,\sigma)$~is false for all~$\lambda$~such that~$\lambda_{0}\leq\lambda\leq\sigma.$~
\end{definition}

Proposition~\ref{p51}~implies that ~$\sigma_{c}(\lambda_{0})>\lambda_{0}.$~Let ~$u(t)$~be any solution to~\eqref{1.1}~
with~$P(u)=0,$~~$M(u)=M(Q),$~
$E(u)/E(Q)\leq\omega_{1}\lambda_{0}^{2}-\omega_{2}\lambda_{0}^{\frac{N(p-1)}{2}}$~
and~$\|\nabla u(0)\|_{2}/\|\nabla Q\|_{2}>1.$~If~$\lambda_{0}>1$~and~$\sigma_{c}=\infty,$~we claim that
there exists a sequence of times~$t_{n}$~such that~$\|\nabla u(t_{n})\|_{2}\rightarrow\infty.$~
In fact, if not, and let ~$\lambda\geq\lambda_{0}$~be such that~
$E(u)/E(Q)=\omega_{1}\lambda^{2}-\omega_{2}\lambda^{\frac{N(p-1)}{2}}.$~
Since there is no sequence~$t_{n}$~such that~$\|\nabla u(t_{n})\|_{2}\rightarrow\infty,$~
there must exists~$\sigma<\infty$~such that
~$\lambda\leq\|\nabla u(t)\|_{2}/\|\nabla Q\|_{2}\leq\sigma$~for all~$t\geq0,$~
which means that~$\exists GB(\lambda,\sigma)$~holds true. Thus ~$\sigma_{c}\leq\sigma<\infty$~
and we get a contradiction.

In view of the above claim, if we can prove that for every~$\lambda_{0}>1$~then~$\sigma_{c}(\lambda_{0})=\infty,$~we then
have in fact proved our Theorem~\ref{th1}.~Thus, in the sequel, we shall carry it out by contradiction; more precisely , fix ~$\lambda_{0}>1$~
and assume~$\sigma_{c}<\infty,$~we shall work toward a absurdity. (It, of course, suffices to do this for ~$\lambda_{0}$~close to 1,
and so we might as well assume that~$\lambda_{0}<(\frac{\omega_{1}}{\omega_{2}})^{\frac{2}{N(p-1)-4}},$
which will be convenient in the sequel.)
For that purpose, we need first to obtain the existence of a critical solution:

\begin{lemma}\label{l81}
$\sigma_{c}(\lambda_{0})<\infty$.~
Then there exist initial data~$u_{c,0}$~and~$\lambda_{c}\in[\lambda_{0},\sigma_{c}(\lambda_{0})]$~such that ~$u_{c}(t)\equiv NLS(t)u_{c,0}$~
is global, ~$P(u_{c})=0,$~~$M(u_{c})=M(Q),$~
~$E(u_{c})/E(Q)=\omega_{1}\lambda_{c}^{2}-\omega_{2}\lambda_{c}^{\frac{N(p-1)}{2}},$~and~
$$\lambda_{c}\leq\frac{\|\nabla u_{c}(t)\|_{2}}{\|\nabla Q\|_{2}}\leq\sigma_{c}\ \ \ for\  all \ \ t\geq 0.$$~

\end{lemma}

We call ~$u_{c}$~ a critical solution since by definition of~$\sigma_{c}$~we have that for all~$\sigma<\sigma_{c}$~
and all~$\lambda_{0}\leq\lambda\leq\sigma,$~$\exists GB(\lambda,\sigma)$ is false, i.e., there are no solutions~$u(t)$~for which
$$P(u)=0,\ \ \ \ M(u)=M(Q),\ \ \ \frac{E(u)}{E(Q)}=\omega_{1}\lambda^{2}-\omega_{2}\lambda^{\frac{N(p-1)}{2}}$$
and$$\lambda\leq\frac{\|\nabla u(t)\|_{2}}{\|\nabla Q\|_{2}}\leq\sigma\ \ \ for\  all \ \ t\geq 0.$$

\begin{proof}
By definition of ~$\sigma_{c},$~there exist sequence~$\lambda_{n}$~and~$\sigma_{n}$~such that~$\lambda_{0}\leq\lambda_{n}\leq\sigma_{n}$~
and $\sigma_{n}\downarrow\sigma_{c}$~for which~$\exists GB(\lambda_{n},\sigma_{n})$~holds. This means that there exists~$u_{n,0}$~
such that~$u_{n}(t)\equiv NLS(t)u_{n,0}$~is global with ~$P(u_{n})=0,$~~$M(u_{n})=M(Q),$~
~$E(u_{n})/E(Q)=\omega_{1}\lambda_{n}^{2}-\omega_{2}\lambda_{n}^{\frac{N(p-1)}{2}},$~and~
$$\lambda_{n}\leq\frac{\|\nabla u_{n}(t)\|_{2}}{\|\nabla Q\|_{2}}\leq\sigma_{n}\ \ \ for\  all \ \ t\geq 0.$$~
The boundedness of~$\lambda_{n}$~ make us pass to a subsequence such that ~$\lambda_{n}$~converges with a limit ~$\lambda'\in[\lambda_{0},\sigma_{c}].$~

According to Lemma~\ref{l64}~where we take~$\phi_{n}=u_{n,0},$~for~$M_{1}+1\leq j\leq M_{2},$~
~$v^{j}(t)\equiv NLS(t)\psi^{j}$~ scatter as~$t\rightarrow+\infty$~and combined with Proposition~\ref{p61},~ for
~$M_{2}+1\leq j\leq M,$~~$v^{j}$~also scatter in one or the other time direction. Thus by the scattering theory, for~$M_{1}+1\leq j\leq M,$~
we have~$E(v_{j})=E(\psi_{j})\geq 0$~and then by~\eqref{6.2}~
$$\sum_{j=1}^{M_{1}}E(\psi^{j})\leq E(\phi_{n})+o_{n}(1).$$
Thus there exists at least one~$1\leq j\leq M_{1}$~with
$$E(\psi^{j})\leq\max(\lim_{n}E(\phi_{n}),0),$$
which, without loss of generality, we might as well take~$j=1.$~
Since, by the profile composition, also
~$M(\psi^{1})\leq \lim_{n}M(\phi_{n})=M(Q),$~we then have
$$\frac{M^{\frac{1-s_{c}}{s_{c}}}(\psi^{1})E(\psi^{1})}{M^{\frac{1-s_{c}}{s_{c}}}(Q)E(Q)}\leq \max\left(\lim_{n}\frac{E(\phi_{n})}{E(Q)},0\right).$$
Thus, there exist~$\tilde{\lambda}\geq\lambda_{0}$~
\footnote{If ~$\lim_{n}E(\phi_{n})\geq 0,$~ we have~$\tilde{\lambda}\geq\lambda'\geq\lambda_{0};$~ while in the case~$\lim_{n}E(\phi_{n})<0,$~
we will have ~$\tilde{\lambda}\geq(\frac{\omega_{1}}{\omega_{2}})^{\frac{2}{N(p-1)-4}}>\lambda_{0}$~ though we might not have~$\tilde{\lambda}\geq\lambda'$~.}such that
$$\frac{M^{\frac{1-s_{c}}{s_{c}}}(\psi^{1})E(\psi^{1})}{M^{\frac{1-s_{c}}{s_{c}}}(Q)E(Q)}
=\omega_{1}\tilde{\lambda}^{2}-\omega_{2}\tilde{\lambda}^{\frac{N(p-1)}{2}}.$$
Note that by Lemma~\ref{l64},~~$v^{1}$~does not scatter, so it follows from Theorem~\ref{t22}~that\\
~$\|\psi^{1}\|^{\frac{1-s_{c}}{s_{c}}}_{2}\|\nabla\psi^{1}\|_{2}<\|Q\|^{\frac{1-s_{c}}{s_{c}}}_{2}\|\nabla Q\|_{2}$~cannot hold.
 Then by Proposition \ref{p21},
we must have $\|\psi^{1}\|^{\frac{1-s_{c}}{s_{c}}}_{2}\|\nabla\psi^{1}\|_{2}\geq\tilde{\lambda}\|Q\|^{\frac{1-s_{c}}{s_{c}}}_{2}\|\nabla Q\|_{2}.$

Now if $\tilde{\lambda}>\sigma_{c}$ and recall that $t_n^1=0,$ then for all  $t$ we know that
\begin{align}\label{8.0}
\tilde{\lambda}^{2}&\leq
\frac{\|v^{1}(t)\|^{2\frac{1-s_{c}}{s_{c}}}_{2}\|\nabla v^{1}(t)\|^{2}_{2}}{\|Q\|^{2\frac{1-s_{c}}{s_{c}}}_{2}\|\nabla Q\|^{2}_{2}}
\leq\frac{\|\nabla v^{1}(t)\|^{2}_{2}}{\|\nabla Q\|^{2}_{2}}
\leq\frac{\sum_{j=1}^{M}\|\nabla v^{j}(t-t_{n}^{j})\|^{2}_{2}+\|\nabla W_{n}^{M}(t)\|^{2}_{2}}{\|\nabla Q\|^{2}_{2}}.
\end{align}
Taking $t=0$ for example, Lemma \ref{l63} implies that
\begin{align*}
\tilde{\lambda}^{2}
&\leq\frac{\sum_{j=1}^{M}\|\nabla v^{j}(-t_{n}^{j})\|^{2}_{2}+\|\nabla W_{n}^{M}\|^{2}_{2}}{\|\nabla Q\|^{2}_{2}}
\leq\frac{\|\nabla u_{n}(0)\|^{2}_{2}}{\|\nabla Q\|^{2}_{2}}+o_{n}(1)\leq \sigma_{c}^{2}+o_{n}(1)
\end{align*}
which contradicts the assumption ~$\tilde{\lambda}>\sigma_{c}$. Hence we must have~$\tilde{\lambda}\leq\sigma_{c}$.~

Now if $\tilde{\lambda}<\sigma_{c},$
we know from the definition of~$\sigma_{c}$~that
~$\exists GB(\tilde{\lambda},\sigma_{c}-\delta)$~is false for any~$\delta>0$~sufficiently small,
and then there exists a nondecreasing sequence~$t_{k}$~of times such that
$$\lim_{k}\frac{\|v^{1}(t_{k})\|_{2}^{\frac{1-s_{c}}{s_{c}}}\|\nabla v^{1}(t_{k})\|_{2}}
{\|Q\|_{2}^{\frac{1-s_{c}}{s_{c}}}\|\nabla Q\|_{2}}\geq \sigma_{c}.$$
Note that~$t_{n}^{1}=0,$~then
\begin{align}\label{8.1}
\sigma_{c}^{2}-o_{k}(1)&\leq
\frac{\|v^{1}(t_{k})\|^{2\frac{1-s_{c}}{s_{c}}}_{2}\|\nabla v^{1}(t_{k})\|^{2}_{2}}{\|Q\|^{2\frac{1-s_{c}}{s_{c}}}_{2}\|\nabla Q\|^{2}_{2}}
\leq\frac{\|\nabla v^{1}(t_{k})\|^{2}_{2}}{\|\nabla Q\|^{2}_{2}}\\ \nonumber
&\leq\frac{\sum_{j=1}^{M}\|\nabla v^{j}(t_{k}-t_{n}^{j})\|^{2}_{2}+\|\nabla W_{n}^{M}(t_{k})\|^{2}_{2}}{\|\nabla Q\|^{2}_{2}}\\ \nonumber
&\leq\frac{\|\nabla u_{n}(t)\|^{2}_{2}}{\|\nabla Q\|^{2}_{2}}+o_{n}(1)\\ \nonumber
&\leq \sigma_{c}^{2}+o_{n}(1),
\end{align}
where by Lemma~\ref{l63}~we take~$n=n(k)$~large. Sending~$k\rightarrow \infty$~and hence ~$n(k)\rightarrow \infty,$~
we conclude that all inequalities must be equalities.
Thus  we conclude that~$W_{n}^{M}(t_{k})\rightarrow0$~
in $H^{1},$ $M(v^{1})=M(Q)$ and  $ v^{j}\equiv0$ for all~$j\geq2. $ Thus easily~$P(v^{1})=P(u_{n})=0.$~
On the other hand if ~$\tilde{\lambda}=\sigma_{c},$   we need not the inductive hypothesis but, similar to \eqref{8.0},  obtain
\begin{align*}
\sigma_{c}^{2}
&\leq\frac{\sum_{j=1}^{M}\|\nabla v^{j}(-t_{n}^{j})\|^{2}_{2}+\|\nabla W_{n}^{M}\|^{2}_{2}}{\|\nabla Q\|^{2}_{2}}
\leq\frac{\|\nabla u_{n}(0)\|^{2}_{2}}{\|\nabla Q\|^{2}_{2}}+o_{n}(1)\leq \sigma_{c}^{2}+o_{n}(1),
\end{align*}
and then again,  we conclude that~$W_{n}^{M}\rightarrow0$~
in $H^{1},$ $M(v^{1})=M(Q)$ and  $ v^{j}\equiv0$ for all~$j\geq2. $
 Moreover, by Lemma~\ref{l63},~for all ~$t$~$$\frac{\|\nabla v^{1}(t)\|^{2}_{2}}{\|\nabla Q\|^{2}_{2}}
 \leq\lim_{n}\frac{\|\nabla u_{n}(t)\|^{2}_{2}}{\|\nabla Q\|^{2}_{2}}\leq\sigma_{c}^{2}.$$
 Hence, we take~$u_{c,0}=v^{1}(0)=\psi^{1}$~and~$\lambda_{c}=\tilde{\lambda}$~to complete our proof.

\end{proof}

\section{Concentration of Critical Solutions and Proof of Theorem~\ref{th1}}

In this section, we will finally prove  Theorem \ref{1.1} by virtue of the precompactness of the flow of the critical solution.
To simplify notation, we take~$u(t)=u_{c}(t)$~in the sequel.

\begin{lemma}\label{l91}
There exists a path~$x(t)$~in~$\mathbb{R}^{N}$~such that
$$K\equiv\left\{u(t,\cdot-x(t))|t\geq 0\right\}\subset H^{1}$$
is precompact in ~$H^{1}.$~

\end{lemma}

\begin{proof}
As is showed in \cite{nonradial} , it suffices to prove that for each sequence of times~$t_{n}\rightarrow\infty,$~
there exists a sequence~$x_{n}$~such that, by passing to a subsequence,  $u(t_{n},\cdot-x_{n})$ converges in  $H^{1}.$

Taking ~$\phi_{n}=u(t_{n})$~in Lemma~\ref{l64} and by definition of $u(t)=u_{c}(t),$
  similar to the proof of Lemma \ref{l81}, we obtain that
there exists at least one~$1\leq j\leq M_{1}$ with
$$E(\psi^{j})\leq\max(\lim_{n}E(\phi_{n}),0).$$
Without loss of generality, we can take~$j=1.$
Since, also
 $M(\psi^{1})\leq \lim_{n}M(\phi_{n})=M(Q),$
 there exist $\tilde{\lambda}\geq\lambda_{0}$
such that
$$\frac{M^{\frac{1-s_{c}}{s_{c}}}(\psi^{1})E(\psi^{1})}{M^{\frac{1-s_{c}}{s_{c}}}(Q)E(Q)}
=\omega_{1}\tilde{\lambda}^{2}-\omega_{2}\tilde{\lambda}^{\frac{N(p-1)}{2}}.$$
Note that by Lemma \ref{l64}, $v^{1}$ does not scatter, so
we must have~$\|\psi^{1}\|_{2}\|\nabla\psi^{1}\|_{2}\geq\tilde{\lambda}\|Q\|_{2}\|\nabla Q\|_{2}.$~\\
Then by the same way as in the proof of Lemma \ref{l81} , we get that
~$W_{n}^{M}(t_{k})\rightarrow0$~
in~$H^{1}$~ and~$ v^{j}\equiv0$~for all~$j\geq2.$~
 Since we know that~$W_{n}^{M}(t)$ is a scattering solution , this implies that
 \begin{equation}\label{8.2}
 W_{n}^{M}(0)=W_{n}^{M}\rightarrow0\ \ \ in\ \  H^{1}.
 \end{equation}
 Consequently, we have
 \begin{equation*}\label{}
u(t_{n})=NLS(-t_{n}^{1})\psi^{1}(x-x_{n}^{1})+W_{n}^{M}(x).
\end{equation*}
Note that by Lemma~\ref{l64},~ $t_{n}^{1}=0,$~and thus
 \begin{equation*}\label{}
u(t_{n},x+x_{n}^{1})=\psi^{1}(x)+W_{n}^{M}(x+x_{n}^{1}).
\end{equation*}
This equality and~\eqref{8.2}~imply our conclusion.
\end{proof}

Using the uniform-in-time ~$H^{1}$~concentration of~$u(t)=u_{c}(t)$~and by changing of variables,
we can easily get
\begin{corollary}\label{c92}
For each~$\epsilon>0,$~there exists~$R>0$~such that for all ~$t,$~
$$\|u(t,\cdot-x(t))\|_{H^{1}(|x|\geq R)}\leq \epsilon.$$
\end{corollary}

With the localization property of $u_{c}$, we show, similar to \cite{holmer10}, that $u_{c}$ must blow up in finite time  using the same method as that
in the proof of Proposition~\ref{p32},
which contradicts the boundedness of~$u_{c}$~in~$H^{1}.$~ Hence,  $u_{c}$ cannot exist and
 $\sigma_{c}=\infty$. As is argued in section 7, this indeed completes the proof of Theorem \ref{th1}.

\appendix

\section{Nonzero Momentum}
Suppose that the solution~$u(x,t)$~with~$M(u)=M(Q),$~~$P(u)\neq0$.~Applying Galilean transform to~$u(x,t),$~we obtain
a new solution~$\tilde{u}(x,t)$:~
$$\tilde{u}(x,t)=e^{ix\cdot\xi_{0}}e^{-it|\xi_{0}|^{2}}u(x-2\xi_{0}t,t).$$
Take~$\xi_{0}=-\frac{P(u)}{M(u)}$~and we get
$$P(\tilde{u})=0,\ \ \ M(\tilde{u})=M(u)=M(Q),\ \ \  \|\nabla\tilde{u}\|_{2}^{2}=\|\nabla u\|_{2}^{2}-\frac{P(u)^{2}}{M(u)}$$
and
\begin{align*}
E(\tilde{u})=\frac{1}{2}\|\nabla u\|_{2}^{2}-\frac{1}{p+1}\|u\|_{p+1}^{p+1}+
\frac{M(u)}{2}\left(\xi_{0}+\frac{P(u)}{M(u)}\right)^{2}-\frac{P(u)^{2}}{2M(u)}=
E(u)-\frac{1}{2}\frac{P(u)^{2}}{M(u)}.
\end{align*}
Thus this choice of ~$\xi_{0}$~make ~$E(\tilde{u})$~attain its lowest value under any choice of~$\xi_{0}\in\mathbb{R}^{N}$.~
And as is stated in~\cite{holmer10},  ~$E(\tilde{u})<E(u)<E(Q)$~implies that we should always implement this transformation to maximize the applicability of
Proposition~\ref{p21}.~

Now what we should do is to show that if the dichotomy of Proposition~\ref{p21}~was already valid for~$u$,~then the selection of case~(1)~versus~(2)~
in Proposition~\ref{p21}~is preserved under the Galilean transformation.

Suppose~$M(u)=M(Q),~E(u)<E(Q)$~and~$P(u)\neq0$.~Define~$\tilde{u}(x,t)$~as above. Let~$\lambda_{-},\lambda$~be defined in terms of ~$E(u)$~
by~\eqref{2.6}~and~$\eta(t)$~in terms of~$u(t)$~by~\eqref{eta}.~Let$\tilde{\lambda}_{-},\tilde{\lambda}$~and~$\tilde{\eta}(t)$~
be the same quantities associated to~$\tilde{u}$.~

Firstly, suppose that case (1) of Proposition~\ref{p21}~holds for~$u$,~which in particular implies that
~$\eta(t)<1$~for all~$t$.~But clearly ~$\tilde{\eta}(t)<\eta(t)<1,$~thus, case (1) of Proposition~\ref{p21}~holds for~$\tilde{u}$~also.

Now conversely, suppose that case (1) of Proposition~\ref{p21}~holds for~$\tilde{u}$,~
then~$\tilde{\eta}(t)^{2}\leq\tilde{\lambda}_{-}^{2}$~for all~$t$.~We claim that
~$$\eta(t)^{2}=\frac{\|\nabla u\|_{2}^{2}}{\|\nabla Q\|_{2}^{2}}
=\tilde{\eta}(t)^{2}+\frac{P(u)^{2}}{2M(u)\|\nabla Q\|_{2}^{2}}
=\tilde{\eta}(t)^{2}+\frac{P(u)^{2}}{2\omega_{1}M(u)E(Q)}
\leq\lambda_{-}^{2}.
$$~
Indeed, this reduced to an algebraic problem now. Denote~$\alpha=\frac{E(u)}{E(Q)}$~and~$\beta=\frac{P(u)^{2}}{M(u)E(Q)}.$~\\
Then ~$\tilde{\lambda}_{-}$~is the smaller root of the equation:
$$\omega_{1}\tilde{\lambda}_{-}^{2}-\omega_{2}\tilde{\lambda}_{-}^{\frac{N(p-1)}{2}}
=\frac{E(\tilde{u})}{E(Q)}=\frac{E(u)}{E(Q)}-\frac{P(u)^{2}}{2M(u)E(Q)}=\alpha-\frac{\beta}{2},
$$
while ~$\lambda_{-}$~is the smaller root of
$$\omega_{1}\lambda_{-}^{2}-\omega_{2}\lambda_{-}^{\frac{N(p-1)}{2}}
=\frac{E(u)}{E(Q)}=\alpha.
$$
Let the function~$f(x)=\omega_{1}x-\omega_{2}x^{\frac{N(p-1)}{4}}.$~
Observe that the above claim follows if we could prove that
~$f(\tilde{\lambda}_{-}^{2}+\frac{\beta}{2\omega_{1}})\leq f(\lambda_{-}^{2}).$~
Equivalently, it suffices to show~$f(\tilde{\lambda}_{-}^{2}+\frac{\beta}{2\omega_{1}})\leq
f(\tilde{\lambda}_{-}^{2})+\frac{\beta}{2},$~or
\begin{align}\label{a}
f(\tilde{\lambda}_{-}^{2}+\frac{\beta}{2\omega_{1}})-
f(\tilde{\lambda}_{-}^{2})\leq\frac{\beta}{2}.
\end{align}
The left hand side of~\eqref{a}~is~$\frac{\beta}{2}-
\omega_{2}\left((\tilde{\lambda}_{-}^{2}+\frac{\beta}{2\omega_{1}})^{\frac{N(p-1)}{4}}-(\tilde{\lambda}_{-}^{2})^{\frac{N(p-1)}{4}}\right)$~
which is certainly no larger than~$\frac{\beta}{2}$~since ~$p-1>\frac{4}{N},$~and we conclude our claim.


\end{document}